\documentclass[]{amsart}

\usepackage{graphicx,color,hyperref}

\usepackage[initials,nobysame]{amsrefs}
\usepackage{amssymb}
\usepackage{subcaption}

\newtheorem{theorem}{Theorem}[section]
\newtheorem{proposition}[theorem]{Proposition}
\newtheorem{corollary}[theorem]{Corollary}
\newtheorem{lemma}[theorem]{Lemma}
\newtheorem{problem}[theorem]{Problem}
\newtheorem{conjecture}[theorem]{Conjecture}
\theoremstyle{definition}
\newtheorem{definition}[theorem]{Definition}

\theoremstyle{remark}
\newtheorem{remark}[theorem]{Remark}

\DeclareMathOperator{\sech}{sech}
\DeclareMathOperator{\cn}{cn}
\DeclareMathOperator{\dn}{dn}
\DeclareMathOperator{\sn}{sn}
\DeclareMathOperator{\am}{am}

\newcommand{\ds}{\partial_s}
\newcommand{\dx}{\partial_x}
\newcommand{\dm}{\partial_m}

\newcommand{\R}{\mathbf{R}}
\newcommand{\Z}{\mathbf{Z}}

\newcommand{\Q}{\mathbf{Q}}

\newcommand{\bS}{\mathbf{S}}

\begin{document}

\title{Elastic curves and self-intersections}
\author[T.~Miura]{Tatsuya Miura}
\address[T.~Miura]{Department of Mathematics, Graduate School of Science, Kyoto University, Kitashirakawa Oiwake-cho, Sakyo-ku, Kyoto 606-8502, Japan}
\email{tatsuya.miura@math.kyoto-u.ac.jp}
\date{\today}

\begin{abstract}
  This is an expository note to give a brief review of classical elastica theory, mainly prepared for giving a more detailed proof of the author's Li--Yau type inequality for self-intersecting curves in Euclidean space. 
  We also discuss some open problems in related topics.
\end{abstract}

\maketitle

\tableofcontents

\section{Introduction}

This is an expository note to supplement and follow up the author's recent work on a Li--Yau type inequality for closed curves in Euclidean space \cite{Miura_LiYau}.
The inequality roughly claims that ``more bending is needed for more multiplicity'', as in the original Li--Yau inequality for surfaces \cite{LiYau}.
Hence it is closely related to elastica theory, a variational problem involving the bending energy for curves.

Variational theory for elastica is initiated by D.\ Bernoulli and L.\ Euler in the 18th century (see e.g.\ \cite{Lev,Tru83} for the history).
An immersed curve $\gamma$ in $\R^n$ is called an \emph{elastica} if it is a critical point of the \emph{bending energy} $B[\gamma]:=\int_\gamma|\kappa|^2ds$ among curves with fixed length $L[\gamma]:=\int_\gamma ds$.
Elastica theory is the oldest variational model of elastic rods (see e.g.\ \cite{Love,LLbook,APbook} for physical contexts), and also the most typical example of geometric higher-order variational problems.
However, this classical theory still continues to evolve in many fields including geometric analysis and the calculus of variations (see e.g.\  \cite{Miura20,miura2024uniqueness} and references therein).

A complete analysis for closed elasticae in Euclidean space is already given by Langer--Singer in the 1980's \cite{LS_JDG,LS_1984_JLMS,LS_85} (see also the lecture notes \cite{Sin}).
In particular, they showed that the only stable closed elastica in $\R^3$ is the one-fold circle \cite{LS_85}.
However, in practice, we can observe many other ``stable'' shapes of closed wires.
Those shapes can also be theoretically treated if one properly takes account of the effect of self-intersections.

The Li--Yau type inequality in \cite{Miura_LiYau} reveals how the bending energy depends on self-intersections at a single point.
Define the \emph{normalized bending energy}
\[
\bar{B}[\gamma]:=L[\gamma]B[\gamma].
\]
We say that $\gamma$ has a point of multiplicity $r$ if there is $P\in\R^n$ such that $\gamma^{-1}(P)$ has at least $r$ distinct points.
Finally, using the complete elliptic integrals of the first kind $K$ and of the second kind $E$, we define a unique parameter $m^*\in(0,1)$ such that $2E(m^*)=K(m^*)$ and then define the universal constant
\[
\varpi^*:=32(2m^*-1)E(m^*)\simeq 28.109...
\]

\begin{theorem}[{\cite{Miura_LiYau}}]\label{thm:intro_Li-Yau}
    Let $n,r\geq2$ be integers.
    Let $\gamma:\R/\Z\to\R^n$ be an immersed closed curve of class $W^{2,2}$ with a point of multiplicity $r$.
    Then
    \[
    \bar{B}[\gamma] \geq \varpi^* r^2.
    \]
    In addition, equality holds if and only if $\gamma$ is a closed $r$-leafed elastica.
    Such a closed $r$-leafed elastica exists if and only if either $n\geq3$ or $r$ is even.
    If $n=2$ and $r$ is odd, then there is $\varepsilon_r>0$ such that
    \[
    \bar{B}[\gamma] \geq \varpi^* r^2 + \varepsilon_r.
    \]
\end{theorem}

In particular, in dimension $n=3$, equality holds for any multiplicity $r\geq2$.
Equality is attained by a new family of curves, which we call \emph{leafed elasticae}.
This is a generalization of Euler's figure-eight elastica; more precisely, an $r$-leafed elastica consists of $r$ leaves of equal length, and each leaf is given by a half-fold figure-eight elastica (Figure \ref{fig:figureeight}).
For example, a closed $2$-leafed elastica is uniquely given by the figure-eight elastica itself.
A closed $3$-leafed elastica is also unique, but given by a new shape (not a classical elastica) which we call the \emph{elastic propeller} (Figure \ref{fig:propeller}).
Remarkably, the elastic propeller can be realized as a stable closed wire in the three-dimensional space (Figure \ref{fig:propeller_photo}), although the figure-eight elastica cannot.

\begin{figure}[htbp]
    \begin{minipage}[b]{0.45\hsize}
        \centering
        \includegraphics[width=40mm]{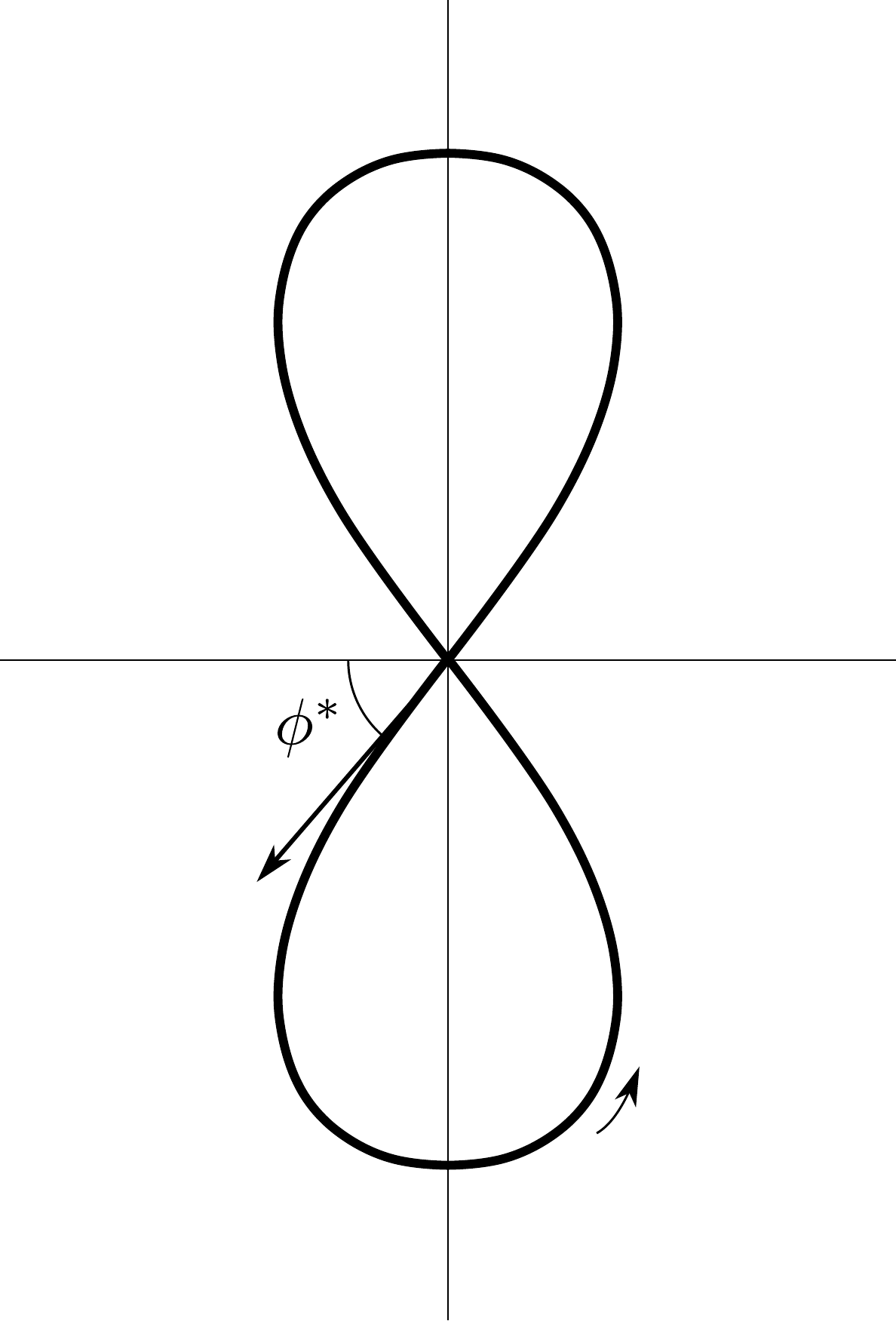}
    \end{minipage}
    \begin{minipage}[b]{0.45\hsize}
        \centering
        \includegraphics[width=20mm]{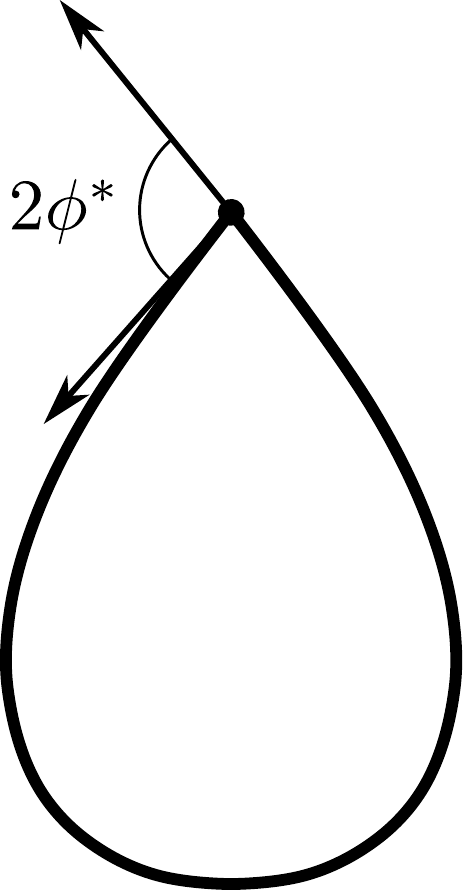}
        \vspace{10mm}
    \end{minipage}
    \caption{Figure-eight elastica (left) and leaf (right). Figures adapted from \cite{Miura_LiYau}.}
    \label{fig:figureeight}
\end{figure}

\begin{figure}[htbp]
    \includegraphics[width=100mm]{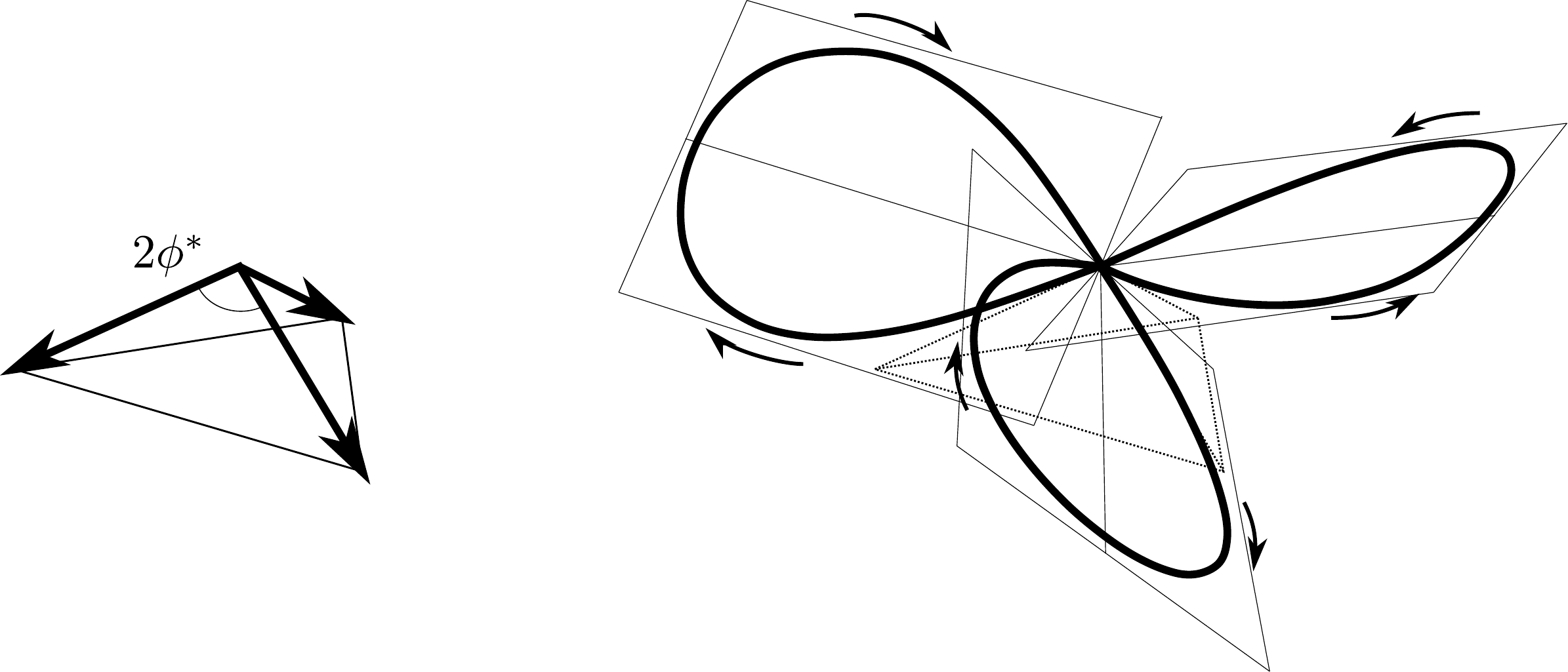}
    \caption{Elastic propeller: The unique closed $3$-leafed elastica \cite{Miura_LiYau}.}
    \label{fig:propeller}
\end{figure}

\begin{figure}[htbp]
  \includegraphics[width=50mm]{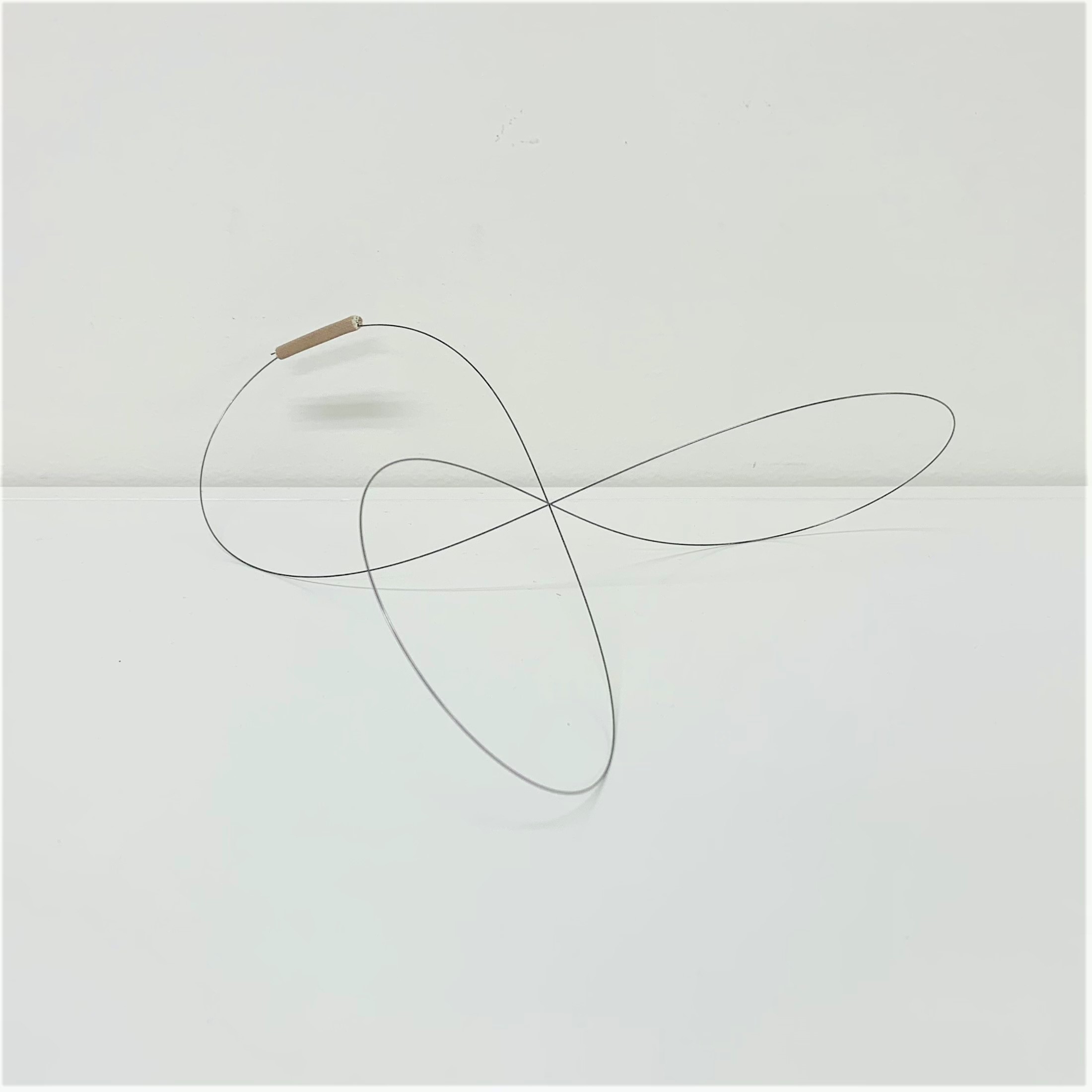}
  \qquad
  \includegraphics[width=50mm]{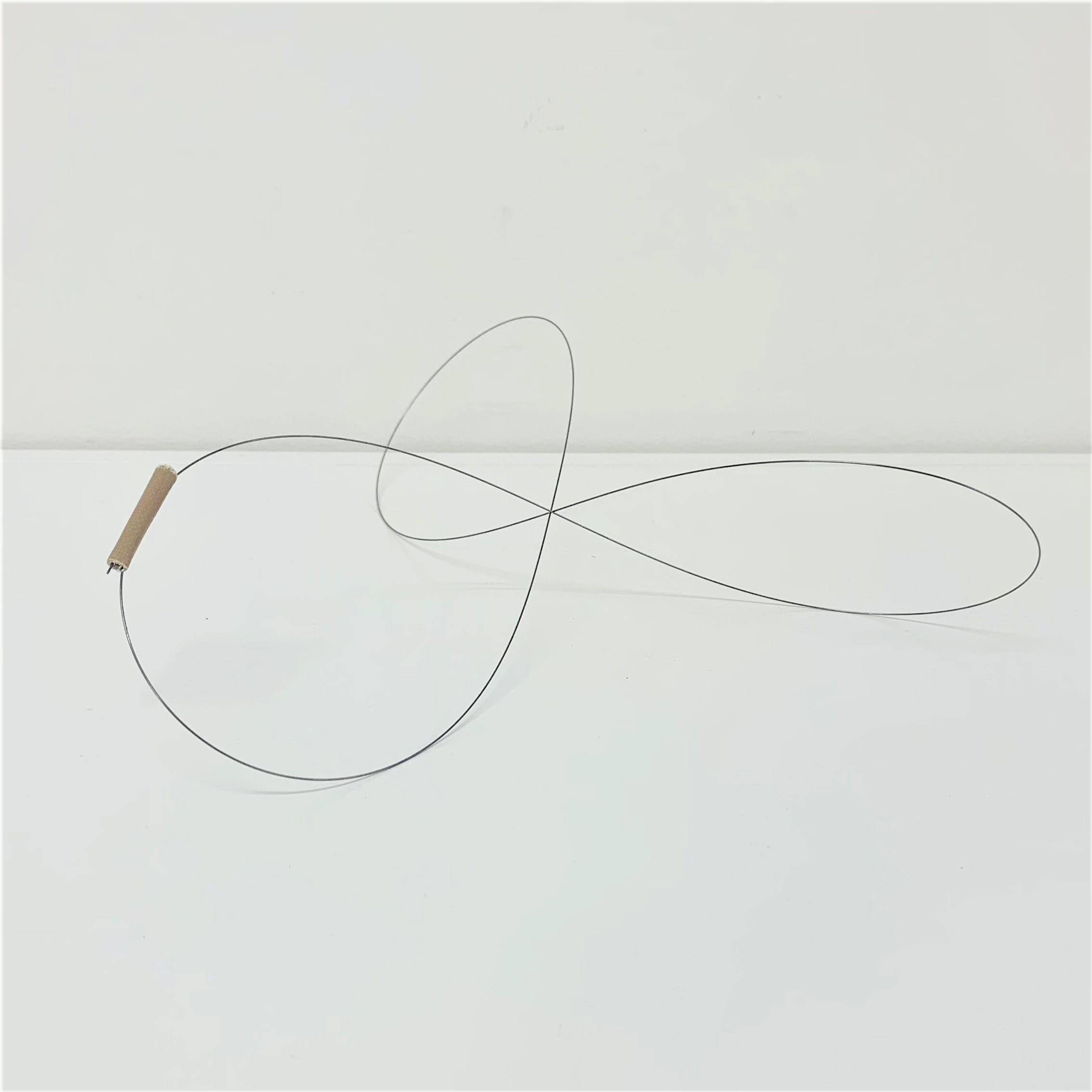}
  \caption{Propeller made of unknotted wire \cite{Miura_LiYau}.}
  \label{fig:propeller_photo}
\end{figure}

An important step towards Theorem \ref{thm:intro_Li-Yau} is to show unique minimality of the leaf, as in Theorem \ref{thm:leaf_minimality} below (corresponding to \cite[Proposition 2.6]{Miura_LiYau}).
The original argument in \cite{Miura_LiYau} crucially relies on previous works about classical elastica theory, so it might not be easy to follow for the readers unfamiliar with this field.

The purpose of this paper is to give a detailed and self-contained proof of Theorem \ref{thm:leaf_minimality}.
To this end, we will briefly review classical elastica theory from both geometric and analytic points of view.
Recall that non-optimal versions of Theorem \ref{thm:intro_Li-Yau} (mostly with $\varpi^*$ replaced by $16$) were previously obtained in several ways \cite{PolPhD,Mosel1998,Wheeler2013,Wojtowytsch2021}, which require no precise understanding of elasticae.
The optimal inequality was first obtained by M\"uller--Rupp \cite{MR23} in the special case $(n,r)=(2,2)$, where planar elastica theory is used.
For the general case we need to use elastica theory in $\R^n$.

This paper requires only modest backgrounds to read; classical differential geometry for curves, ODE theory including the Picard-Lindel\"{o}f theorem, and basic tools in functional analysis involving Sobolev spaces.
Hence, this paper would also be useful as a beginner's guide for students and researchers who are interested in elastica theory.

This paper is organized as follows.
Sections \ref{sec:elastica_geom} and \ref{sec:elastica_anal} are reviews on classical elastica theory.
Section \ref{sec:elastica_geom} focuses on geometric aspects, while Section \ref{sec:elastica_anal} is more analytic.
Section \ref{sec:Li-Yau} completes the proof of Theorem \ref{thm:leaf_minimality}, and then discusses some open problems around the Li--Yau type inequality.
In Appendix \ref{sec:elliptic} we review fundamental properties of Jacobi elliptic integrals and functions for the reader's convenience.

\subsection*{Acknowledgements}
This work is supported by JSPS KAKENHI Grant Numbers JP21H00990, JP23H00085, and JP24K00532.

\section{Geometric perspective on elastica theory}\label{sec:elastica_geom}

We first review some fundamental results in classical elastica theory within the class of smooth curves in Euclidean space (cf.\ \cite[Section 1]{Sin}).
More analytical results involving weak derivatives will be discussed separately in the next section.

\subsection{Bending energy}

Let $I=(a,b)\subset\mathbf{R}$ be an open interval and $\bar{I}:=[a,b]$.
Let $\gamma:\bar{I}\to\mathbf{R}^n$ be a smooth curve, that is, a smooth map (say, of class $C^\infty$) with non-vanishing derivative $|\gamma'(x)|>0$ for any $x\in\bar{I}$.
Here and in the sequel the dimension $n\geq2$ is arbitrary if not specified.
Define the length functional
\begin{equation}\label{eq:def_length}
  L[\gamma]:=\int_I ds,
\end{equation}
where $ds$ is understood as the arclength measure $ds:=|\gamma'|dx$.
For a differentiable map 
$\psi:\bar{I}\to\mathbf{R}^d$ with any $d\geq1$, we define the arclength derivative along $\gamma$ by 
$$\ds\psi=\psi_s:=\frac{1}{|\gamma'|}\psi'.$$
In particular, throughout this manuscript, let $T:\bar{I}\to\R^n$ denote the unit tangent
$$T:=\ds\gamma,$$
and $\kappa:\bar{I}\to\R^n$ denote the curvature vector
$$\kappa:=\ds^2\gamma.$$
The \emph{bending energy} of $\gamma$ is then defined by
\begin{equation}\label{eq:def_bending}
  B[\gamma] := \int_I |\kappa|^2ds.
\end{equation}
This quantity is also called the total squared curvature, the Euler--Bernoulli energy, Euler's elastic energy and so on.
[Note: Sometimes the bending energy is simply referred to as ``elastic energy'' but this might be (physically) confusing because the bending energy does not encompass other elastic effects like stretching or twisting.]

We emphasize that the bending energy is different from the simple second-order Dirichlet energy $\int_I|\gamma''|^2dx$ unless $|\gamma'|\equiv1$.
In fact, we compute
\begin{equation}
  \ds\gamma=\frac{1}{|\gamma'|}\gamma', \quad \ds^2\gamma=\frac{1}{|\gamma'|^2}\gamma''-\frac{\langle\gamma',\gamma''\rangle}{|\gamma'|^4}\gamma',
\end{equation}
and hence
\begin{equation}\label{eq:def_bending2}
  B[\gamma] = \int_I \left| \frac{1}{|\gamma'|^2}\gamma''-\frac{\langle\gamma',\gamma''\rangle}{|\gamma'|^4}\gamma' \right|^2|\gamma'|dx
  = \int_I \left(\frac{|\gamma''|^2}{|\gamma'|^3}- \frac{\langle\gamma',\gamma''\rangle^2}{|\gamma'|^5} \right)dx.
\end{equation}

The bending energy $B$ and the length $L$ have important geometric invariances: For any Euclidean isometry $\Phi:\mathbf{R}^n\to\mathbf{R}^n$ and reparametrization $\xi:\bar{J}\to \bar{I}$ (with $|\xi'|>0$) the quantity is preserved, namely,
$$B[\Phi\circ\gamma\circ\xi]=B[\gamma], \quad L[\Phi\circ\gamma\circ\xi]=L[\gamma].$$
Recall that any Euclidean isometry is given by an affine map $\Phi(x)=Ax+b$, where $A\in O(n)$ is an orthogonal matrix and $b\in \mathbf{R}^n$.

Note however that the energies are not preserved by dilation.
For example, if $C_r$ is a planar circle of radius $r$, then $|\kappa|=1/r$ and hence $B[C_r]=(1/r)^2L[C_r]=2\pi /r$.
In general, we have
\begin{equation}
  B[\Lambda\gamma]=\frac{1}{\Lambda}B[\gamma], \quad L[\Lambda\gamma]= \Lambda L[\gamma] \quad \mbox{for $\Lambda>0$}.
\end{equation}
In particular, in the class of closed curves $\gamma$, the bending energy has no minimizer attaining $\inf_\gamma B=0$, and even no critical point since the rescaling $(1+\varepsilon)\gamma$ gives an energy-decreasing perturbation for $\varepsilon>0$.

On the other hand, the above scaling property implies that the normalized bending energy
\begin{equation}
  \bar{B}[\gamma]:=L[\gamma]B[\gamma]=L\int_\gamma|\kappa|^2ds
\end{equation}
is scaling invariant in the sense that
\begin{equation}
  \bar{B}[\Lambda \gamma]=\bar{B}[\gamma] \quad \mbox{for $\Lambda>0$}.
\end{equation}
The normalized bending energy has a nontrivial minimum among closed curves attained only by a round circle.

\begin{theorem}
  Let $\gamma$ be a smooth closed curve in $\mathbf{R}^n$.
  Then
  \begin{equation}
    \bar{B}[\gamma] \geq 4\pi^2,
  \end{equation}
  where equality holds if and only if $\gamma$ is a planar round circle.
\end{theorem}

\begin{proof}
  We define the total (absolute) curvature by
  \begin{equation}
    TC[\gamma]:=\int_\gamma|\kappa|ds.
  \end{equation}
  The Cauchy--Schwarz inequality $\|f\|_{L^2}\|g\|_{L^2}\geq (f,g)_{L^2}$ and Fenchel's theorem $TC[\gamma]\geq2\pi$ imply that
  \begin{equation*}
    \bar{B}[\gamma]=L\int_0^L|\kappa|^2ds \geq \left(\int_0^L|\kappa|ds\right)^2 = TC[\gamma]^2 \geq 4\pi^2.
  \end{equation*}
  In addition, if $\bar{B}[\gamma]=4\pi^2$, then equality in the first (Cauchy--Schwarz) inequality implies that $|\kappa|$ is constant, while the second (Fenchel) implies that $\gamma$ is convex and planar.
  Thus, $\bar{B}[\gamma]=4\pi^2$ holds if and only if $\gamma$ is a planar round circle.
\end{proof}

An immediate corollary is

\begin{corollary}
  In the class of smooth closed curves in $\mathbf{R}^n$ of length $L>0$, the bending energy $B$ is minimized by a round circle of radius $\frac{L}{2\pi}$, which is a unique minimizer up to invariances (isometries and reparametrizations).
\end{corollary}

Now we turn to more general problems; local minimizers, critical points, other boundary value problems and so on.
In order to address these problems, an important first step is to investigate the first variation of the bending energy, and also of the length in view of the multiplier method.
Then we obtain the corresponding Euler--Lagrange equation.

\subsection{Elastica: The Euler--Lagrange equation}

Now we formally derive the Euler--Lagrange equation satisfied by critical points of the bending energy among smooth curves subject to the fixed-length constraint.
In view of the standard multiplier method (later justified in the framework of Sobolev spaces), if a smooth curve $\gamma:\bar{I}\to\R^n$ is such a length-constrained critical point, then there exists $\lambda\in\mathbf{R}$ such that $\gamma$ is a critical point of the energy $B+\lambda L$ without length constraint.
Hence at least for any linear perturbation $\gamma_\varepsilon:=\gamma+\varepsilon\eta$ with $\eta\in C^\infty_c(I;\mathbf{R}^n)$,
$$\frac{d}{d\varepsilon}\big(B[\gamma_\varepsilon]+\lambda L[\gamma_\varepsilon]\big)\Big|_{\varepsilon=0}=0.$$
Here the class $C^\infty_c$ means $C^\infty$ functions compactly supported on $I$.

Let $\gamma:\bar{I}\to\mathbf{R}^n$ be a smooth curve.
For a perturbation $\gamma_\varepsilon:=\gamma+\varepsilon\eta$ with $\eta\in C^\infty(\bar{I};\mathbf{R}^n)$ we calculate the first variation by using formulae \eqref{eq:def_length} and \eqref{eq:def_bending2}:
\begin{align}
  \frac{d}{d\varepsilon}L[\gamma_\varepsilon]\Big|_{\varepsilon=0} &=  \int_I \frac{d}{d\varepsilon}|\gamma'+\varepsilon\eta'| \Big|_{\varepsilon=0} dx =  \int_I\frac{\langle\gamma',\eta'\rangle}{|\gamma'|}dx, \label{eq:FVLgeneral}\\
  \frac{d}{d\varepsilon}B[\gamma_\varepsilon]\Big|_{\varepsilon=0} &= \int_I \left(\frac{2\langle\gamma'',\eta''\rangle}{|\gamma'|^3}-\frac{3|\gamma''|^2\langle\gamma',\eta'\rangle}{|\gamma'|^5} \right. \label{eq:FVBgeneral}\\
  & \qquad \left. - \frac{2\langle\gamma',\gamma''\rangle\big( \langle\gamma',\eta''\rangle+\langle\gamma'',\eta'\rangle \big)}{|\gamma'|^5} + \frac{5\langle\gamma',\gamma''\rangle^2\langle\gamma',\eta'\rangle}{|\gamma'|^7} \right)dx. \nonumber
\end{align}

Now we express the formulae in terms of the weighted measure $ds=|\gamma'|dx$ on $I$.
We deduce from \eqref{eq:FVLgeneral} and \eqref{eq:FVBgeneral} that
\begin{align}
  \frac{d}{d\varepsilon}L[\gamma_\varepsilon]\Big|_{\varepsilon=0} &= \int_I\langle\gamma_s,\eta_s\rangle ds, \label{eq:FVLunit}\\
  \frac{d}{d\varepsilon}B[\gamma_\varepsilon]\Big|_{\varepsilon=0} &= \int_I \left(2\langle\gamma_{ss},\eta_{ss}\rangle-3|\gamma_{ss}|^2\langle\gamma_s,\eta_s\rangle \right)ds. \label{eq:FVBunit}
\end{align}
These expressions are nothing but formulae \eqref{eq:FVLgeneral} and \eqref{eq:FVBgeneral} under the unit-speed assumption $|\gamma'|\equiv1$ and the resulting orthogonality $(\gamma',\gamma'')=(\frac{1}{2}|\gamma'|^2)'\equiv0$.

Noting that the standard integration-by-parts procedure works for $ds$ since
$$\int_a^b f_sgds=\int_a^b f'gdx=[fg]_a^b-\int_a^bfg'dx=[fg]_a^b-\int_a^b fg_sds,$$
we further deduce that for $\eta\in C^\infty(\bar{I};\mathbf{R}^n)$,
\begin{align}
  \frac{d}{d\varepsilon}L[\gamma_\varepsilon]\Big|_{\varepsilon=0} &= -\int_I\langle\gamma_{ss},\eta\rangle ds -[\langle\gamma_s,\eta\rangle]_a^b, \label{eq:FVLintegralbyparts_BC}\\
  \frac{d}{d\varepsilon}B[\gamma_\varepsilon]\Big|_{\varepsilon=0} &= \int_I \big( 2\langle\gamma_{ssss},\eta\rangle+ 6\langle\gamma_{ss},\gamma_{sss}\rangle\langle\gamma_s,\eta\rangle + 3|\gamma_{ss}|^2\langle\gamma_{ss},\eta\rangle \big)ds \label{eq:FVBintegralbyparts_BC}\\
  & \qquad\qquad + [2\langle\gamma_{ss},\eta_s\rangle]_a^b - [\langle2\gamma_{sss}+3|\gamma_{ss}|^2\gamma_s,\eta\rangle]_a^b. \nonumber
\end{align}

In particular, if $\eta\in C^\infty_c(I;\mathbf{R}^n)$, then $\eta$ and $\eta_s$ vanish at the endpoints and hence
\begin{align}
  \frac{d}{d\varepsilon}L[\gamma_\varepsilon]\Big|_{\varepsilon=0} &= -\int_I\langle\gamma_{ss},\eta\rangle ds, \label{eq:FVLintegralbyparts}\\
  \frac{d}{d\varepsilon}B[\gamma_\varepsilon]\Big|_{\varepsilon=0} &= \int_I \big\langle 2\gamma_{ssss}+6\langle\gamma_{ss},\gamma_{sss}\rangle\gamma_{s}+3|\gamma_{ss}|^2\gamma_{ss},\eta \big\rangle ds. \label{eq:FVBintegralbyparts}
\end{align}
By \eqref{eq:FVLintegralbyparts} and \eqref{eq:FVBintegralbyparts}, and by the fundamental lemma of the calculus of variations, we deduce the key ODE of fourth order:
\begin{equation}\label{eq:elastica_ODE_fourthorder}
  2\gamma_{ssss}+ 6\langle\gamma_{ss},\gamma_{sss}\rangle\gamma_{s} + 3|\gamma_{ss}|^2\gamma_{ss} -\lambda \gamma_{ss} = 0.
\end{equation}

\begin{definition}[Elastica]
  A smooth curve $\gamma:\bar{I}\to\mathbf{R}^n$ is called an \emph{elastica} if $\gamma$ solves equation \eqref{eq:elastica_ODE_fourthorder} for some $\lambda\in\mathbf{R}$.
  When we specify the multiplier $\lambda$, it is also referred to as \emph{$\lambda$-elastica}.
\end{definition}

An elastica is also called Euler's elastica, Euler--Bernoulli elastica, elastic curve and so on in the literature.
Throughout this manuscript we mainly use the term {\em elastica} (and {\em elasticae} for plural).
Accordingly, we call \eqref{eq:elastica_ODE_fourthorder} the {\em elastica equation}.

[Note: Euler's original work is written in Latin and he used ``elasticum/elastica''.
However, the wording ``elastica/elasticae'' is commonly used in the recent English literature. This manuscript follows the latter. The author would like to thank Marius M\"uller for pointing out this fact.]

Now we can already deduce that any elastica is analytic by using the Cauchy--Kovalevskaya theorem for the system of ODEs.
We will anyway prove the analyticity by directly showing that all elasticae are represented by (analytic) elliptic integrals and functions.

\subsection{Uniqueness and dimensional rigidity}

Standard uniqueness theory for the elastica equation already implies a strong dimensional rigidity.

\begin{theorem}\label{thm:uniquness_dim_rigidity}
    Let $n\geq2$ and $\gamma:\bar{I}\to\R^n$ be an elastica.
    Let $x_0\in\bar{I}$
    and define
    $$d:=\dim\mathrm{span}\{\gamma_s(x_0),\gamma_{ss}(x_0),\gamma_{sss}(x_0)\}.$$
    Then the image of $\gamma$ is contained in a $d$-dimensional affine subspace.
\end{theorem}

\begin{proof}
    Let $\tilde{\gamma}:[0,L]\to\R^n$ be the arclength parametrization of $\gamma$, and let $s_0:=\int_a^{x_0}|\gamma'|$, where we write $I=(a,b)$.
    Then
    $$d=\dim\mathrm{span}\{\tilde{\gamma}'(s_0),\tilde{\gamma}''(s_0),\tilde{\gamma}'''(s_0)\},$$
    and $\tilde{\gamma}$ solves \eqref{eq:elastica_ODE_fourthorder} with the standard derivatives, that is, for some $\lambda\in\R$,
    \begin{equation}\label{eq:elastica_ODE_arclength}
        2\tilde{\gamma}'''' + 6\langle\tilde{\gamma}'',\tilde{\gamma}'''\rangle\tilde{\gamma}' + 3|\tilde{\gamma}''|^2\tilde{\gamma}''-\lambda\tilde{\gamma}'' = 0.
    \end{equation}
    Now we can choose an isometry $\Phi:\R^n\to\R^n$, which is necessarily affine, such that the arclength parametrized curve 
    $$\xi:=\Phi\circ\tilde{\gamma}:[0,L]\to\R^n$$ 
    solves the same equation as \eqref{eq:elastica_ODE_arclength} with the initial condition
    \begin{equation*}
        \text{$\xi(s_0)=0$ and $\xi'(s_0),\xi''(s_0),\xi'''(s_0)\in\R^d\times\{0\}\subset\R^n$.}
    \end{equation*}
    On the other hand, by the Picard--Lindel\"of theorem, there exists a unique $d$-dimensional solution
    $$\zeta:[0,L]\to\R^d$$
    to equation \eqref{eq:elastica_ODE_arclength} with the initial condition
    $$\text{$\zeta^{(i)}(s_0)=\iota^{-1}\circ\xi^{(i)}(s_0)$ for $i=0,1,2,3$,}$$
    where $\iota:\R^d\to\R^d\times\{0\}\subset\R^n$ denotes the canonical injection.
    This map defines an $n$-dimensional solution $\iota\circ\zeta:[0,L]\to\R^n$ to equation \eqref{eq:elastica_ODE_arclength} with the same initial condition as $\xi$.
    By uniqueness, we have
    $$\iota\circ\zeta=\xi,$$
    and hence $\tilde{\gamma}=\Phi^{-1}\circ\iota\circ\zeta$, which means that the assertion holds true.
\end{proof}

In particular, we have the following

\begin{corollary}\label{cor:uniqueness_dim_rigidity}
  Let $\gamma:\bar{I}\to\mathbf{R}^n$ be an elastica.
  Then the following assertions hold.
  \begin{enumerate}
    \item The map $\dim\mathrm{span}\{\gamma_s,\gamma_{ss},\gamma_{sss}\}:\bar{I}\to\{1,2,3\}$ is constant.
    \item If $d$ denotes the above constant value, then the image of $\gamma$ is contained in a $d$-dimensional affine subspace.
    \item If $\kappa=0$ at some point, then $\gamma$ is at most two-dimensional.
  \end{enumerate}
\end{corollary}


Therefore, without loss of generality we may only consider elasticae in $\R^3$.
From now on we say that an elastica is \emph{planar} if the image is contained in a (two-dimensional) plane, and otherwise \emph{non-planar} or \emph{spatial}.

\subsection{Curvature equation in general dimensions}

The fourth order differential system in \eqref{eq:elastica_ODE_fourthorder} is still hard to handle.
We can further reduce it by using geometric notions.
Define the normal derivative $\nabla_s$ of a vector field $\psi$ along $\gamma$ by
$$\nabla_s\psi:=(\ds\psi)^\perp=\ds\psi-\langle\ds\psi,T\rangle T,$$
where $\perp$ means the normal projection.
Then we have

\begin{lemma}[First variation: Geometric form]\label{lem:first_variation_normal}
  For a smooth curve $\gamma:\bar{I}\to\mathbf{R}^n$ with $I=(a,b)$, formulae \eqref{eq:FVLintegralbyparts_BC} and \eqref{eq:FVBintegralbyparts_BC} are also represented as follows:
  \begin{align}
    \frac{d}{d\varepsilon}L[\gamma_\varepsilon]\Big|_{\varepsilon=0} &= -\int_I\langle\kappa,\eta\rangle ds +[\langle T,\eta\rangle]_a^b, \label{eq:FVLintegralbyparts_BC_curvature}\\
    \frac{d}{d\varepsilon}B[\gamma_\varepsilon]\Big|_{\varepsilon=0} &= \int_I \langle2\nabla_s^2\kappa + |\kappa|^2\kappa,\eta\rangle ds + [2\langle\kappa,\eta'\rangle]_a^b - [\langle2\partial_s\kappa+3|\kappa|^2T,\eta\rangle]_a^b. \label{eq:FVBintegralbyparts_BC_curvature}
  \end{align}
\end{lemma}

\begin{proof}
  We only address the second integral as it is the only nontrivial part.
  By definition $\kappa=\gamma_{ss}$ and direct computations we obtain
  \begin{align*}
    2\nabla_s^2\kappa+|\kappa|^2\kappa
    = 2\gamma_{ssss}-2\langle\gamma_{ssss},\gamma_s\rangle\gamma_s-2\langle\gamma_{sss},\gamma_s\rangle\gamma_{ss}+|\gamma_{ss}|^2\gamma_s,
  \end{align*}
  and hence it is now sufficient to check that
  $$-2\langle\gamma_{ssss},\gamma_s\rangle\gamma_s-2\langle\gamma_{sss},\gamma_s\rangle\gamma_{ss} = 6\langle\gamma_{ss},\gamma_{sss}\rangle\gamma_s + 2|\gamma_{ss}|^2\gamma_{ss}. $$
  This follows since the orthogonality of $\gamma_s$ and $\gamma_{ss}$ implies $6\langle\gamma_{ss},\gamma_{sss}\rangle
  +2\langle\gamma_{ssss},\gamma_s\rangle=2\langle\gamma_s,\gamma_{ss}\rangle_{ss}=0$ and $2|\gamma_{ss}|^2+2\langle\gamma_{sss},\gamma_s\rangle=2\langle\gamma_{s},\gamma_{ss}\rangle_s=0$.
\end{proof}

\begin{remark}[Normal derivative]
    The readers who are not familiar with geometric variational problems may wonder why it is important to introduce the normal derivative.
    This idea works well because the first variations only depend on the normal variations, except at the endpoints.
    This is a common feature in geometric variational problems due to invariance with respect to automorphisms.
    Indeed, in our case, the functionals under consideration are invariant along any (nonlinear) smooth perturbation of the form $\gamma_\varepsilon=\gamma\circ\xi_\varepsilon$, where $\xi_\varepsilon$ denotes a family of reparametrizations.
    If we consider $\xi_\varepsilon=s+\varepsilon\zeta+o(\varepsilon)$ for $\zeta\in C^\infty_c(I)$, then the perturbation can be expanded as $\gamma_\varepsilon=\gamma+\varepsilon\zeta\gamma'+o(\varepsilon)$, and hence $\frac{d}{d\varepsilon}\big|_{\varepsilon=0}L[\gamma+\varepsilon\zeta\gamma']=\frac{d}{d\varepsilon}\big|_{\varepsilon=0}L[\gamma_\varepsilon]=0$; the same is true for $B$.
    The arbitrariness of $\zeta$ implies that the first variations must vanish in the tangent direction $\gamma'$.
\end{remark}

Lemma \ref{lem:first_variation_normal} immediately implies an alternative form of the elastica equation:

\begin{proposition}[Elastica equation in terms of curvature vector]
  Let $\lambda\in\R$.
  A smooth curve $\gamma:\bar{I}\to\mathbf{R}^n$ is a $\lambda$-elastica if and only if it solves
  \begin{equation}\label{eq:elastica_ODE_general_dim}
    2\nabla_s^2\kappa + |\kappa|^2\kappa -\lambda\kappa =0.
  \end{equation}
\end{proposition}

We now turn to solving this equation for $n=2,3$.
The main technical difference between $n=2$ and $n=3$ is that any smooth planar curve admits the canonical orthonormal frame, but this is not the case for a spatial curve.

\subsection{Curvature equation for planar elasticae}

For a planar curve $\gamma:\bar{I}\to\mathbf{R}^2$ we can always take the orthonormal frame $\{T,N\}$ by
$$T:=\gamma_s, \quad N:=R_{\pi/2}T,$$
where $R_{\pi/2}$ denotes the counterclockwise rotation matrix through angle $\pi/2$.
Under this choice of $N$ we can uniquely define the \emph{signed curvature}
$$k:=\langle\gamma_{ss},N\rangle.$$
Then, since $\gamma_{ss}=\kappa$ is parallel to $N$, that is, $\langle\gamma_{ss},\gamma_s\rangle=0$,
\begin{equation}\label{eq:planar''}
  \gamma_{ss}=\kappa=kN.
\end{equation}
Also, since $N_s=R_{\pi/2}T_s=kR_{\pi/2}N=-kT$, we find the relation
\begin{equation}\label{eq:planar'''}
  \gamma_{sss}=k_sN-k^2T,
\end{equation}
and moreover
\begin{align}\label{eq:planar_normal}
  \nabla_s\kappa =(\gamma_{sss})^\perp = k_sN, \quad \nabla_s^2\kappa =((k_sN)_s)^\perp = k_{ss}N.
\end{align}
Now the elastica equation \eqref{eq:elastica_ODE_general_dim} can completely be reduced to a scalar ODE:

\begin{proposition}[Planar elastica equation]\label{prop:elastica_ODE_2D}
  A planar curve $\gamma:\bar{I}\to\mathbf{R}^2$ is $\lambda$-elastica if and only if its signed curvature $k:\bar{I}\to\mathbf{R}$ satisfies
  \begin{equation}\label{eq:elastica_ODE_2D}
    2k_{ss}+k^3-\lambda k=0.
  \end{equation}
\end{proposition}

\begin{proof}
  We obtain \eqref{eq:elastica_ODE_2D} by \eqref{eq:elastica_ODE_general_dim} in view of \eqref{eq:planar''} and \eqref{eq:planar_normal}.
\end{proof}

Since a planar curve is determined by the curvature up to an isometry, the solvability of the elastica equation is now essentially reduced to the above scalar second order ODE.

\subsection{Curvature equation for spatial elasticae}

Now we consider a non-planar elastica $\gamma:\bar{I}\to\mathbf{R}^3$.
For the unit tangent $T=\gamma_s$ we have
$$|T_s|=|\kappa|>0 \quad \text{on $\bar{I}$}$$
since otherwise $\gamma$ is contained in a plane by Theorem \ref{thm:uniquness_dim_rigidity}.
In particular, the (non-vanishing) scalar curvature $k$ and the unit normal $N$ are well-defined by 
$$k:=|\kappa|, \quad N:=\frac{\kappa}{|\kappa|}.$$
In addition, the elastica $\gamma$ has non-vanishing torsion
$$t:=\langle N_s,B\rangle\neq0 \quad \text{on $\bar{I}$},$$ 
where $B$ denotes the unit binormal
$$B:=T\times N,$$ 
since again otherwise $\gamma$ is planar by Theorem \ref{thm:uniquness_dim_rigidity}.
Therefore we can associate the standard Frenet--Serret frame $\{T,N,B\}$ to any spatial elastica.

Recall the Frenet--Serret formula:
\begin{equation}
  \begin{pmatrix}
    T_s \\ N_s \\ B_s
  \end{pmatrix}
  =
  \begin{pmatrix}
    0 & k & 0\\
    -k & 0 & t\\
    0 & -t & 0
  \end{pmatrix}
  \begin{pmatrix}
    T \\ N \\ B
  \end{pmatrix}.
\end{equation}
In particular,
\begin{align}
  \gamma_{ss} &=\kappa=kN, \label{eq:spatial''}\\
  \gamma_{sss} &=k_sN-k^2T+ktB. \label{eq:spatial'''}
\end{align}
Computing the normal derivatives, we obtain
\begin{align}
  \nabla_s\kappa &= (\gamma_{sss})^\perp = k_sN + ktB, \nonumber\\
  \nabla_s^2\kappa &= ((k_sN + ktB)_s)^\perp = (k_{ss}-kt^2)N + (2k_st + kt_s)B, \label{eq:spatial_normal}
\end{align}
and thus the following expression equivalent to \eqref{eq:elastica_ODE_general_dim}:
$$(2k_{ss}+k^3-kt^2-\lambda k)N + 2(2k_st+kt_s)B=0.$$
Therefore, there is an undetermined constant $c\neq0$ such that
\begin{align}\label{eq:elastica_ODE_3D}
  \begin{cases}
    2k_{ss}+k^3-\lambda k- 2kt^2 =0,\\
    k^2t = c\neq0.
  \end{cases}
\end{align}
By the second equation it needs to hold that $k>0$ and $t\neq0$ everywhere.
Hence, by deleting $t$ in the first equation by using the second, the system is essentially reduced to a scalar ODE.

\begin{proposition}[Spatial elastica equation]\label{prop:elastica_ODE_3D}
    Let $\gamma:\bar{I}\to\mathbf{R}^3$ be a smooth curve that is not contained in any plane.
    Then $\gamma$ is a $\lambda$-elastica if and only if $\gamma$ has non-vanishing curvature $k>0$ and torsion $t\neq0$ on $\bar{I}$, and there exists $c\in\R\setminus\{0\}$ such that $k$ and $t$ satisfy
    \begin{align}\label{eq:elastica_ODE_3D'}
      \begin{cases}
        2k_{ss}+k^3-\lambda k- \dfrac{2c^2}{k^3} =0,\\
        t = \dfrac{c}{k^2}.
      \end{cases}
    \end{align}
\end{proposition}


Once the curvature and the torsion are determined, the corresponding arclength parametrized spatial curve is again uniquely determined up to isometry thanks to the Frenet--Serret formula.

\subsection{Explicit formula for curvature}

Now we obtain explicit solutions to the above equations for the curvature $k$.
To this end we need to use Jacobi elliptic integrals and functions.
The definitions are given in Appendix \ref{sec:elliptic}, where one can also find fundamental properties necessary for reading this manuscript.
Here we only emphasize that throughout this paper we will use the parameter $m\in[0,1]$ instead of the elliptic modulus $\sqrt{m}$.

Since equation \eqref{eq:elastica_ODE_2D} can formally be regarded as the case $c=0$ for the first equation in \eqref{eq:elastica_ODE_3D'}, we may focus on the equation of the form
$$2k''+k^3-\lambda k- \dfrac{2c^2}{k^3} =0.$$
Multiplying $k'$ and integrating the equation yield that for some constant $A$,
$$ (k')^2 + \frac{1}{4}k^4 - \frac{\lambda}{2} k^2 + \frac{c^2}{k^2} = A.$$
Multiplying $4k^2$ and letting $u:=k^2$, we obtain
\begin{equation}\label{eq:u=k^2_ODE}
 (u')^2 = - u^3 + 2\lambda u^2 + 4Au - 4c^2 =: P(u).
\end{equation}
This means that in both planar and spatial cases the squared curvature satisfies the same form of ODE with the cubic polynomial $P$, which is classically known to be solvable by elliptic functions.

The following fact is exactly what we need.

\begin{lemma}[Elliptic function and polynomial ODE]\label{lem:polynomial}
  Let $P$ be a cubic polynomial with real roots $\alpha_1\leq 0\leq\alpha_2<\alpha_3$ of the form $$P(x)=(\alpha_1-x)(\alpha_2-x)(\alpha_3-x).$$
  Let $u:\bar{I}\to[\alpha_2,\alpha_3]$ be a $C^2$ solution to the ODE:
  \begin{equation}\label{eq:ODE_3rd_polynomial}
    (u')^2=P(u).
  \end{equation}
  Then either $u$ is constant, namely $u\equiv\alpha_2$ or $u\equiv\alpha_3$, or there is $s_0\in\mathbf{R}$ such that
  \begin{align}\nonumber
    u(s) =\alpha_3-(\alpha_3-\alpha_2)\sn^2\left(\frac{\sqrt{\alpha_3-\alpha_1}}{2}s+s_0 , \frac{\alpha_3-\alpha_2}{\alpha_3-\alpha_1} \right).
  \end{align}
\end{lemma}

The proof is not difficult.
Since the above first-order ODE is locally uniquely solvable, up to the choice of the sign $\pm u'=\sqrt{P(u)}$, whenever $u$ takes values strictly between $\alpha_2$ and $\alpha_3$, the only delicate point is whether the graph of $u$ touches the lines $u=\alpha_2$ and $u=\alpha_3$ in a non-isolated way.
Such a behavior is however prevented by our $C^2$ assumption.
(Note that if $\alpha_1=\alpha_2=0$ the solution $u$ never touches the line $u=0$.) 
Details are left for the reader's exercise.
See also \cite{Davis1962} for general backgrounds on the relation between polynomial equations and elliptic functions.

We are now in a position to obtain a unified explicit formula for curvature.

\begin{theorem}[Explicit formula for curvature]\label{thm:curvature_formula}
    Let $\lambda\in\R$.
    Let $\gamma:[0,L]\to\R^3$ be an arclength parametrized $\lambda$-elastica.
    Suppose that $\gamma$ is not a straight segment.
    Then there is a unique choice of four parameters $m,w,A,c\in\R$ such that
    $$0\leq m\leq w\leq 1, \quad w>0, \quad A>0,$$
    and such that the squared curvature of $\gamma$ is represented by
    \begin{equation}\label{eq:curvature_formula}
      |\kappa(s)|^2=A^2\left( 1-\frac{m}{w}\sn^2\Big(\frac{A}{2\sqrt{w}}s+s_0 , m \Big) \right)
    \end{equation}
    for some $s_0\in\mathbf{R}$,
    \begin{equation}\label{eq:det_formula}
        \det\big( \gamma_s(s)\ \gamma_{ss}(s)\ \gamma_{sss}(s) \big) = c,
    \end{equation}
    \begin{equation}\label{eq:lambda_formula}
      \lambda = \frac{A^2}{2w}(3w-m-1),
    \end{equation}
    and
    \begin{equation}\label{eq:c_formula}
      4c^2 = \frac{A^6}{w^2}(1-w)(w-m).
    \end{equation}
\end{theorem}

\begin{remark}\label{rem:determinant_c}
    If $\gamma$ is a planar elastica, then by Corollary \ref{cor:uniqueness_dim_rigidity} 
    \[
    \det\big( \gamma_s(s)\ \gamma_{ss}(s)\ \gamma_{sss}(s) \big) =0,
    \]
    while if $\gamma$ is a spatial elastica, then the torsion $t$ is well-defined and
    \[
    \det\big( \gamma_s(s)\ \gamma_{ss}(s)\ \gamma_{sss}(s) \big) = |\kappa(s)|^2t(s)\equiv\textrm{const}.
    \]
\end{remark}

\begin{proof}[Proof of Theorem \ref{thm:curvature_formula}]
  We first independently address the case that the $\lambda$-elastica $\gamma$ has constant curvature $|\kappa|$; the constant is non-zero by assumption.
  Note first (in general) $c\in\R$ is uniquely determined by \eqref{eq:c_formula} (with Remark \ref{rem:determinant_c}).
  By the constant-curvature assumption (and $A>0$), equation \eqref{eq:curvature_formula} holds if and only if $m=0$.
  Then $A>0$ is uniquely determined by \eqref{eq:curvature_formula}.
  Now $w$ is also uniquely determined by \eqref{eq:lambda_formula}.
  Using the elastica equation we can directly verify that those unique parameters $m,w,A,c$ also satisfy \eqref{eq:c_formula}.

  In what follows we assume that $|\kappa|$ is not constant.
  The squared curvature $u:=|\kappa|^2$ satisfies equation \eqref{eq:u=k^2_ODE} with a cubic polynomial $P$.
  Then $P(0)=-4c^2\leq0$ and $\lim_{x\to\pm\infty}P(x)=\mp\infty$.
  In addition, since $u$ is non-constant, $P(x)>0$ at some $x>0$.
  Therefore, the cubic polynomial $P$ has roots
  $$P(x)=(\alpha_1-x)(\alpha_2-x)(\alpha_3-x)$$
  such that
  $$\alpha_1\leq0\leq\alpha_2<\alpha_3.$$
  Using the relation between roots and coefficients, we have
  \begin{align*}
    2\lambda &= \alpha_1+\alpha_2+\alpha_3,\\
    4A &= -(\alpha_1\alpha_2+\alpha_2\alpha_3+\alpha_3\alpha_1),\\
    -4c^2 &= \alpha_1\alpha_2\alpha_3.
  \end{align*}
  In addition, the function $u=|\kappa|^2$ takes values within $[\alpha_2,\alpha_3]$ by equation \eqref{eq:u=k^2_ODE} together with nonnegativity of $(u')^2$ and $u$.
  Then we can solve the equation in terms of an elliptic function by Lemma \ref{lem:polynomial} by using $\alpha_1,\alpha_2,\alpha_3$.
  Rewriting the constants $\alpha_1,\alpha_2,\alpha_3$ in terms of new (unique) parameters $A>0$ and $0<m\leq w\leq 1$ by
  \begin{align*}
    \alpha_1 = A^2\left( 1-\frac{1}{w} \right), \quad \alpha_2 =A^2\left(1-\frac{m}{w}\right), \quad \alpha_3 = A^2,
  \end{align*}
  we find that there is some $s_0\in\mathbf{R}$ such that $u=|\kappa|^2$ is given by \eqref{eq:curvature_formula} with relations \eqref{eq:lambda_formula} and \eqref{eq:c_formula}.
\end{proof}

Langer--Singer's diagram (\cite[Fig.\ 1]{LS_1984_JLMS}, \cite[Fig.\ 1]{LS_85}, \cite[Figure 5]{Sin}) provides a concise and intuitive classification of all configurations of elasticae in $\R^3$ using the $(m,w)$-coordinate.
[Note: Langer--Singer use the elliptic modulus $p:=\sqrt{m}\in[0,1]$ instead of the parameter $m\in[0,1]$, and also $w^2$ instead of $w$.]

\subsection{Explicit formula for planar elasticae}

With the general formula for curvature at hand, we can integrate and parametrize all elasticae in terms of elliptic integrals and functions.
In this paper we mainly address the planar case for simplicity.

\begin{theorem}\label{thm:curvature_formula_2D}
    Let $\lambda\in\R$.
    Let $\gamma:[0,L]\to\R^2$ be an arclength parametrized planar $\lambda$-elastica.
    Then the signed curvature $k$ is given by either of the following forms:
    \begin{enumerate}
        \item $k(s)=0$,
        \item $k(s)=\pm A\cn(\alpha s+\beta ,m )$ for some $A>0$, $\alpha>0$, $\beta\in\R$, $m\in(0,1)$ such that $A^2=4\alpha^2m$ and $2m\lambda = A^2(2m-1)$,
        \item $k(s)=\pm A\sech(\alpha s+\beta)$ for some $A>0$, $\alpha>0$, $\beta\in\R$ such that $A^2=4\alpha^2=2\lambda$,
        \item $k(s)=\pm A\dn(\alpha s+\beta , m )$ for some $A>0$, $\alpha>0$, $\beta\in\R$, $m\in(0,1)$ such that $A^2=4\alpha^2$ and $2\lambda = A^2(2-m)$,
        \item $k(s)=\pm\sqrt{\lambda}$, where $\lambda>0$.
    \end{enumerate}
\end{theorem}

\begin{proof}
    It is an almost direct corollary of Theorem \ref{thm:curvature_formula}.
    Since the constant-curvature case is obvious, we assume that a given elastica has non-constant curvature.
    Then the elastica corresponds to some case with $m>0$ in Theorem \ref{thm:curvature_formula}.
    Since the elastica is planar, we have $c=0$, and hence either $m=w$ or $w=1$.
    
    The former case $m=w$ with $m\in(0,1)$ corresponds to case (ii); by Theorem \ref{thm:curvature_formula} the curvature is of the form
    \[
    k^2=A^2(1-\sn^2(\alpha s+\beta,m))=A^2\cn^2(\alpha s+\beta,m),
    \]
    and hence the smoothness of $k$ (with the fact that the derivative of $\cn$ does not vanish at the zeroes) implies that 
    \[
    k=\pm A\cn(\alpha s+\beta,m),
    \]
    where the constants satisfy all the necessary relations.
    
    The latter case $w=1$ with $m\in(0,1)$ corresponds to case (iv); again by Theorem \ref{thm:curvature_formula},
    \[
    k^2=A^2(1-m\sn^2(\alpha s+\beta,m))=A^2\dn^2(\alpha s+\beta,m)
    \]
    and since $\dn$ is positive we have
    \[
    k=\pm A\dn^2(\alpha s+\beta,m).
    \]
    
    Finally, the case $m=1$, where we also need to have $w=1$, corresponds to case (iii) since $\cn(x,1)=\dn(x,1)=\sech{x}$.
\end{proof}


Now we obtain explicit parametrizations of planar elasticae.
The integrability of planar elasticae dates back to a work of Saalsch\"utz in 1880 \cite{Saa1880}.
For a planar curve $\gamma$, let $\theta$ denote the tangential angle function so that $\ds\gamma=T=(\cos\theta,\sin\theta)$.
The choice of continuous $\theta$ is unique modulo $2\pi$.
Recall that $\ds\theta=k$ since $\ds T=\ds\theta(-\sin\theta,\cos\theta)=kN$.

\begin{theorem}[Explicit parametrizations of planar elasticae]\label{thm:planar_explicit}
  Let $\gamma:[0,L]\to\R^2$ be an arclength parametrized elastica.
  Then, up to similarity (i.e., translation, rotation, reflection, and dilation), the curve is represented by
  $$\gamma(s)=\gamma_*(s+s_0)$$
  for some $s_0\in\R$, where $\gamma_*$ is one of the following arclength parametrized curves:
  \begin{itemize}
    \item \textup{(Case I --- Linear elastica)}
    \begin{equation}\label{eq:EP1}\notag
      \gamma_\ell(s) =
      \begin{pmatrix}
      s  \\
      0
      \end{pmatrix}.
    \end{equation}
    In this case, $\theta_\ell(s)=0$ and $k_\ell(s)=0$.
    \item \textup{(Case II --- Wavelike elastica)}
    There exists $m \in (0,1)$ such that
    \begin{equation}\label{eq:EP2}\notag
      \gamma_w(s)=\gamma_w(s,m) =
      \begin{pmatrix}
      2 E(\am(s,m),m )-  s  \\
      -2 \sqrt{m} \cn(s,m)
      \end{pmatrix}.
    \end{equation}
    In this case, $\theta_w(s)=2\arcsin(\sqrt{m}\sn(s,m))$ and $k_w(s) = 2 \sqrt{m} \cn(s,m).$
    \item \textup{(Case III --- Borderline elastica)}
    \begin{equation}\label{eq:EP3}\notag
      \gamma_b(s) =
      \begin{pmatrix}
      2 \tanh{s} - s  \\
      - 2 \sech{s}
      \end{pmatrix}.
    \end{equation}
    In this case, $\theta_b(s)=2\arcsin(\tanh{s})$ and $k_b(s) = 2\sech{s}.$
    \item \textup{(Case IV --- Orbitlike elastica)}
    There exists $m \in (0,1)$ such that
    \begin{equation}\label{eq:EP4}\notag
      \gamma_o(s)=\gamma_o(s,m) = \frac{1}{m}
      \begin{pmatrix}
      2 E(\am(s,m),m)  + (m-2)s \\
      - 2\dn(s,m)
      \end{pmatrix}.
    \end{equation}
    In this case, $\theta_o(s)=2\am(s,m)$ and $k_o(s) = 2 \dn(s,m).$
    \item \textup{(Case V --- Circular elastica)}
    \begin{equation}\label{eq:EP5}\notag
      \gamma_c(s) =
      \begin{pmatrix}
      \sin{s}  \\
      -\cos{s}
      \end{pmatrix}.
    \end{equation}
    In this case, $\theta_c(s)=s$ and $k_c(s)=1$.
  \end{itemize}
\end{theorem}

\begin{proof}
    By Theorem \ref{thm:curvature_formula_2D} we deduce that, up to a dilation, the curvature of any planar elastica is given by one of Cases I--V.
    Direct computations show that the derivatives of the given formulae satisfy $\partial_s\gamma_*=(\cos\theta_*,\sin\theta_*)$ and $\partial_s\theta_*=k_*$.
\end{proof}

\begin{remark}
    Although the above proof is logically complete, it does not tell anything about how to derive the formulae.
    There are several ways of derivations.
    A direct argument via integration of elliptic functions can be found in \cite[Section 4.3]{MYarXiv2203}, including more general $p$-elasticae (the classical elastica corresponds to $p=2$).
    Another argument via a dynamical system formulation can be found in \cite[Apprendix B.3]{MR23}.
\end{remark}

\begin{remark}
    All the parametrizations are chosen so that $\theta_*$ is odd (and hence $\theta_*(0)=0$ and $k_*$ is even) and $k_*(0)\geq0$, as well as $\langle\gamma_*(0),e_1\rangle=0$ and $\langle\gamma_*(0),e_2\rangle\leq0$, where $\{e_j\}_{i=1}^n$ denotes the canonical basis of $\R^n$.
    Figure \ref{fig:2-Elastiace} describes vertically flipped versions of those parametrizations.
\end{remark}

\begin{figure}[htbp]
\centering
\includegraphics[width=125mm]{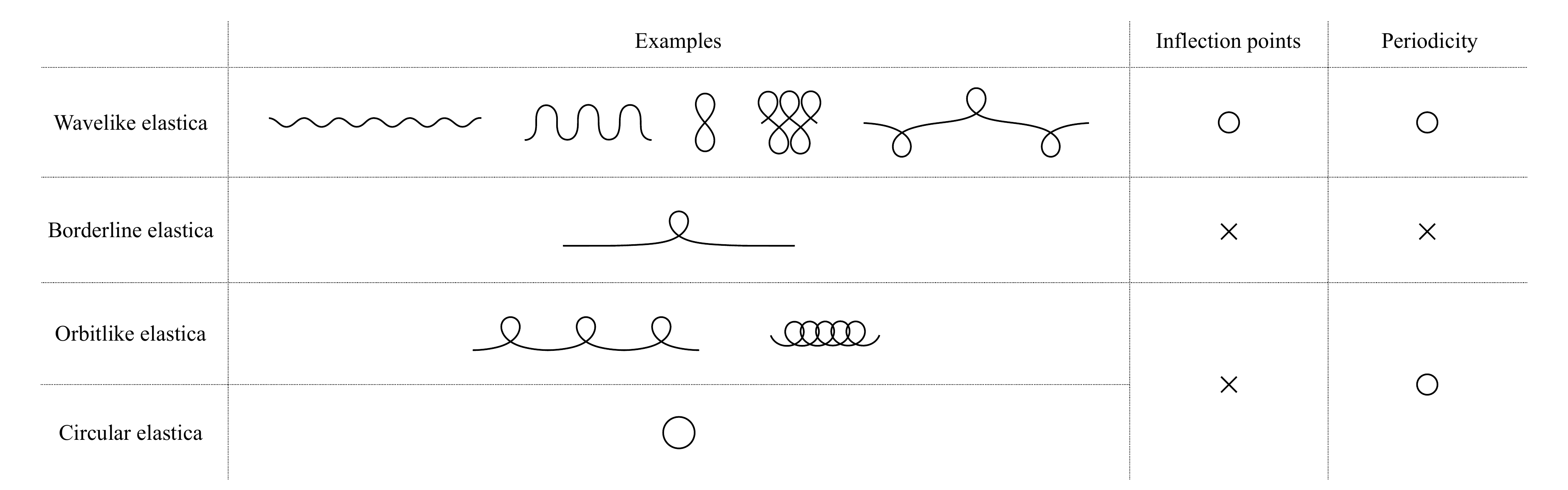}
\caption{Basic patterns of planar elasticae \cite{MY_2024_Crelle}.}
\label{fig:2-Elastiace}
\end{figure}

\begin{remark}
    The tangential angle of the borderline elastica is twice the Gudermannian function, which has several equivalent expressions:
    $$\theta_b(s)=2\am(s,1)=2\arcsin(\tanh{s})=2\arctan(\sinh{s})=4\arctan(e^s)-\pi.$$
    It is easy to observe that the function is strictly increasing from $-\pi$ to $\pi$.
\end{remark}

\begin{remark}\label{rem:multiplier}
    We can directly compute the multipliers of the above curves:
    \begin{itemize}
        \item[(I)] $\gamma_\ell$ is a $\lambda$-elastica for any $\lambda\in\R$.
        \item[(II)] $\gamma_w=\gamma_w(\cdot,m)$ is a $\lambda$-elastica for $\lambda=2(2m-1)$.
        \item[(III)] $\gamma_b$ is a $\lambda$-elastica for $\lambda=2$.
        \item[(IV)] $\gamma_o=\gamma_o(\cdot,m)$ is a $\lambda$-elastica for $\lambda=2(2-m)$.
        \item[(V)] $\gamma_c$ is a $\lambda$-elastica for $\lambda=1$.
    \end{itemize}
    The sign of $\lambda$ is preserved under similarity actions (although the value is not).
    In particular, the case of \emph{free elastica} ($\lambda=0$) only occurs in case (I) and case (II) with $m=\frac12$.
    The latter, unique nontrivial free elastica is called the \emph{rectangular elastica}, which is represented by the graph of a periodic graph with vertical slopes at the inflection points.
\end{remark}

\begin{remark}\label{rem:periodicity}
    Any wavelike elastica has periodicity in the sense that the curvature $k_w$ and the tangential angle $\theta_w$ are both anti-periodic with anti-period $2K(m)$, and the curve $\gamma_w$ is quasi-periodic in the sense that
    \begin{align}\label{eq:wave_periodicity}
        \gamma_w(s+4K(m),m)=\gamma_w(s,m) + 4
        \begin{pmatrix}
          2 E(m)-  K(m)  \\
          0
        \end{pmatrix}.
    \end{align}
    As for orbitlike elastica, the curvature $k_o$ is $2K(m)$-periodic, while $\theta_o$ and $\gamma_o$ are quasi-periodic in the sense that $\theta_o(s+2K(m)) = \theta_o(s) + 2\pi$ and
    \begin{align}\label{eq:orbit_periodicity}
        \gamma_o(s+2K(m),m) = \gamma_o(s,m) + 
        \frac{2}{m}
          \begin{pmatrix}
          2 E(m)  + (m-2)K(m) \\
          0
          \end{pmatrix}.
    \end{align}
\end{remark}

We close this section by classifying closed planar elasticae.

In view of periodicity \eqref{eq:wave_periodicity} and \eqref{eq:orbit_periodicity}, we investigate the behavior of some quantities involving complete elliptic integrals.
We first address the wavelike case.

\begin{lemma}\label{lem:elliptic_integral_wavelike}
  The function $2E(m)-K(m)$ is strictly decreasing from $\frac{\pi}{2}$ to $-\infty$ with respect to $m\in(0,1)$.
  In particular, there is a unique root $m^*\in(0,1)$ such that $2E(m^*)-K(m^*)=0$.
\end{lemma}

\begin{proof}
  The monotonicity follows by the facts that $E'<0$ with $E(0)=\frac{\pi}{2}$ and $E(1-)=1$, and that $K'>0$ with $K(0)=\frac{\pi}{2}$ and $K(1-)=\infty$.
\end{proof}

In particular, by \eqref{eq:wave_periodicity} we find that the above parameter $m^*$ gives a nontrivial closed elastica (Figure \ref{fig:figureeight}).

\begin{definition}[Figure-eight elastica]\label{def:figure-eight_closed}
  A closed elastica is called a \emph{figure-eight elastica} if, up to similarity and reparametrization, the curve is given by the wavelike elastica $\gamma_w(\cdot+s_0,m^*)$ for some $s_0\in\R$.
\end{definition}

We now turn to the orbitlike case.

\begin{lemma}\label{lem:elliptic_integral_orbitlike}
  The function $2E(m)+(m-2)K(m)$ is strictly decreasing from $0$ to $-\infty$ with respect to $m\in(0,1)$.
  In particular, $2E(m)+(m-2)K(m)<0$ for all $m\in(0,1)$.
\end{lemma}

\begin{proof}
    Using $E'=\frac{E-K}{2m}$ and $K'=\frac{E-(1-m)K}{2m(1-m)}$, we compute
    \begin{align*}
        \Big(2E+(m-2)K\Big)' = -\frac{E-(1-m)K}{2(1-m)} = -mK'<0.
    \end{align*}
    Considering the values at $m=0$ and $m\to1$ completes the proof.
\end{proof}

We are now ready to state and prove the classification theorem.

\begin{theorem}[Classification of closed planar elasticae]\label{thm:closed_planar}
  Let $\gamma:\R/\Z\to\R^2$ be an elastica.
  Then $\gamma$ is either a circle or a figure-eight, possibly multiply covered.
\end{theorem}

\begin{proof}
  Up to similarity and reparametrization we may assume that the elastica under consideration is of the form $\gamma(s)=\gamma_*(s+s_0)$, where $s_0\in\R$ and $\gamma_*$ is one of the five parametrizations in Theorem \ref{thm:planar_explicit}.
  It is clear that in case V the only possibility is a circle, and in cases I and III there are no candidates.

  Consider case II.
  By the $4K(m)$-periodicity of the curvature $k_w$, the length of $\gamma$ needs to be of the form $L=4K(m)\mu$ for some positive integer $\mu$.
  In addition we need to have $\gamma(L)=\gamma(0)$, and hence $2E(m)-K(m)=0$ by periodicity \eqref{eq:wave_periodicity}.
  Lemma \ref{lem:elliptic_integral_wavelike} implies that $m=m^*$, that is, $\gamma$ is a $\mu$-fold figure-eight elastica.

  Finally we prove that $\gamma$ does not belong to case IV.
  Similarly as above, by the $2K(m)$-periodicity of $k_o$ and periodicity \eqref{eq:orbit_periodicity}, we need to have $2E(m)+(m-2)K(m)=0$, but this does not occur by Lemma \ref{lem:elliptic_integral_orbitlike}.
  The proof is complete.
\end{proof}

\begin{remark}
    In fact, the same classification as Theorem \ref{thm:closed_planar} holds for a slightly wider class, namely for elasticae $\gamma:[a,b]\to\R^2$ such that $\gamma(a)=\gamma(b)$ and $\gamma_s(a)=\gamma_s(b)$; in particular, periodicity of the curvature is not necessarily assumed.
    A proof is given in \cite[Theorem 5.6]{MYarXiv2203} for more general $p$-elasticae.
\end{remark}

\begin{remark}[Spatial elasticae]
    The integrability of spatial elasticae is more technical, but Langer--Singer discovered the general integrability in 1984 \cite{LS_1984_JLMS}, more than 100 years later than the planar case.
    They obtained explicit parametrizations of spatial elasticae, which are particularly used for classification of closed elasticae.
    In addition, they also developed a beautiful min-max argument \cite{LS_85} to deduce the following remarkable classification result:
    Any closed elastica $\gamma:\R/\Z\to\R^3$ is given by either
    \begin{enumerate}
        \item a circle (planar),
        \item a figure-eight elastica (planar), or
        \item one of an infinite family of embedded torus knots (spatial),
    \end{enumerate}
    possibly multiply-covered; among them, the only stable elastica is the one-fold circle.
    Here we do not enter the details and leave it to the original papers \cite{LS_1984_JLMS,LS_85} as well as the lecture notes \cite{Sin}.
\end{remark}


\section{Analytic perspective on elastica theory}\label{sec:elastica_anal}

From this section we use standard tools in functional analysis around Sobolev spaces; all the facts we will use can be found in most of standard textbooks, see e.g.\ \cite{Brezis2011,Leoni2017}.

Let $m\geq1$ be an integer and $1\leq p\leq\infty$.
Let $W^{m,p}(I;\R^n)$ denote the function space of vector-valued Sobolev functions such that all derivatives of up to $m$-th order are $L^p$ integrable.
Recall that for $p>1$ the Sobolev space $W^{m,p}(I;\R^n)$ is compactly embedded in $C^{m-1}(\bar{I};\R^n)$ \cite[Theorem 8.8]{Brezis2011}.
In particular, if $m\geq2$, then any $\gamma\in W^{m,p}(I;\R^n)$ is of class $C^1(\bar{I};\R^n)$, so that immersedness is still well-defined:
$$W^{m,p}_\mathrm{imm}(I;\R^n):= \left\{\gamma\in W^{m,p}(I;\R^n) \,\left|\, \min_{x\in\bar{I}}|\gamma'(x)|>0 \right.\right\}.$$
For elastica theory, the most natural energy space is $W^{2,2}_\mathrm{imm}(I;\R^n)$.
It is easy to check that for each $\gamma\in W^{2,2}_\mathrm{imm}(I;\R^n)$ the length $L[\gamma]$ and the bending energy $B[\gamma]$ are still well-defined.



\subsection{Existence of minimizers: Direct method}

Thanks to the Sobolev framework, the existence of minimizers under various boundary conditions follows by the direct method in the calculus of variations.
To formulate this result we impose the so-called \emph{clamped boundary condition}, which geometrically prescribes the curve at the endpoints up to first order (as proposed by D.\ Bernoulli). 
Given $L_0>0$, $P_0,P_1\in\R^n$ such that $|P_0-P_1|<L_0$, and $V_0,V_1\in\bS^{n-1}\subset\R^n$, let
\[
\mathcal{A}:=\left\{\gamma\in W^{2,2}_\mathrm{imm}(I;\R^n)\,\left|\,
\begin{aligned}
    L[\gamma]=L_0,\ \gamma(a)=P_0,\ \gamma(b)=P_1,\\
    \gamma_s(a)=V_0,\ \gamma_s(b)=V_1
\end{aligned}
\right.\right\}.
\]
Note that the condition $|P_0-P_1|<L_0$ is natural to impose; if $|P_0-P_1|>L_0$ then $\mathcal{A}$ is empty, while if $|P_0-P_1|=L_0$ then $\mathcal{A}$ is not empty only when $\mathcal{A}$ consists of the trivial straight segment.

\begin{theorem}[Existence of clamped minimizers]\label{thm:existence_minimizer}
    Let $L_0>0$, $P_0,P_1\in\R^n$ such that $|P_0-P_1|<L_0$, and $V_0,V_1\in\bS^{n-1}\subset\R^n$.
    Then there exists $\bar{\gamma}\in\mathcal{A}$ such that
    \[
    B[\bar{\gamma}] = \inf_{\gamma\in\mathcal{A}} B[\gamma].
    \]
\end{theorem}

To prove this it is important to check that the arclength reparametrization also works in the Sobolev framework.

\begin{lemma}\label{lem:arclength_reparametrization_Sobolev}
    Let $\gamma\in W^{2,2}_\mathrm{imm}(I;\R^n)$.
    Then there exists $\sigma\in W^{2,2}(I)$ such that 
    \begin{enumerate}
        \item $\sigma$ is a $C^1$-diffeomorphism from $\bar{I}$ to $\bar{J}$, where $J:=(0,L[\gamma])$;
        \item $\tilde{\gamma}:=\gamma\circ\sigma^{-1}$ is an element of $W^{2,2}(J;\R^n)$;
        \item $|\tilde{\gamma}'|=1$ a.e.\ in $J$.
    \end{enumerate}
    In addition,
    \[
    B[\gamma] = B[\tilde{\gamma}] = \int_J |\tilde{\gamma}''(\tilde{x})|^2 d\tilde{x}, \quad L[\gamma] = L[\tilde{\gamma}] = \int_J d\tilde{x}.
    \]
\end{lemma}

\begin{proof}
    Let $\sigma(t):=\int_0^t|\gamma'(x)|dx$ for $t\in\bar{I}=[a,b]$.
    Then $\sigma\in W^{2,2}(I)\subset C^1(\bar{I})$, and also $\sigma$ is a $C^1$-diffeomorphism from $\bar{I}$ to $\bar{J}$ since $\sigma(a)=0$, $\sigma(b)=L[\gamma]$, and $\sigma'=|\gamma'|>0$.
    Define $\tilde{\gamma}:=\gamma\circ\sigma^{-1}$.
    Then, since $\gamma\in W^{1,2}(I;\R^n)$ and $\sigma^{-1}\in C^1(\bar{J})$, the standard change of variables works so that $\tilde{\gamma}\in W^{1,2}(J;\R^n)$ and 
    \[
    \tilde{\gamma}'=(\sigma^{-1})'(\gamma'\circ\sigma^{-1}) = \frac{\gamma'\circ\sigma^{-1}}{|\gamma'\circ\sigma^{-1}|}.
    \]
    In particular, $|\tilde{\gamma}'|=1$.
    We finally prove that $\tilde{\gamma}\in W^{2,2}(J;\R^n)$.
    Since $\gamma'\in W^{1,2}(I;\R^n)$, the change of variables also works for $\gamma'\circ\sigma^{-1}$ so that $\gamma'\circ\sigma^{-1}\in W^{1,2}(J;\R^n)$ and 
    \[
    (\gamma'\circ\sigma^{-1})'=\frac{\gamma''\circ\sigma^{-1}}{|\gamma'\circ\sigma^{-1}|}.
    \]
    In addition, since $\min_{\bar{I}}|\gamma'|>0$, we also have $|\gamma'\circ\sigma^{-1}|^{-1}\in W^{1,2}(J)$ and \[
    \left(\frac{1}{|\gamma'\circ\sigma^{-1}|}\right)'= - \frac{\langle \gamma'\circ\sigma^{-1},(\gamma'\circ\sigma^{-1})'\rangle}{|\gamma'\circ\sigma^{-1}|^3} = - \frac{\langle \gamma'\circ\sigma^{-1},\gamma''\circ\sigma^{-1}\rangle}{|\gamma'\circ\sigma^{-1}|^4}.
    \]
    By the Leibniz rule, the second derivative $\tilde{\gamma}''$ exists (in $L^1$) and is given by
    \[
    \tilde{\gamma}'' = \left(\frac{\gamma'\circ\sigma^{-1}}{|\gamma'\circ\sigma^{-1}|}\right)' = \frac{\gamma''\circ\sigma^{-1}}{|\gamma'\circ\sigma^{-1}|^2} - \frac{\langle \gamma'\circ\sigma^{-1},\gamma''\circ\sigma^{-1}\rangle}{|\gamma'\circ\sigma^{-1}|^4}\gamma'\circ\sigma^{-1}.
    \]
    Since $|\gamma'|,|\gamma'|^{-1}\in C(\bar{I})\subset L^\infty(I)$, we have
    $
    |\tilde{\gamma}''|^2 \leq C|\gamma''\circ \sigma^{-1}|^2
    $
    and hence
    \begin{align*}
        \int_J|\tilde{\gamma}''(\tilde{x})|^2 d\tilde{x} &\leq C\int_J|\gamma''\circ \sigma^{-1}(\tilde{x})|^2 d\tilde{x}\\
        &= C\int_I|\gamma''(x)|^2\frac{1}{|\gamma'(x)|}dx \quad (x:=\sigma^{-1}(\tilde{x})) \\
        &\leq C'\int_I|\gamma''(x)|^2dx < \infty.
    \end{align*}
    This implies that $\tilde{\gamma}''\in L^2(J;\R^n)$.
    The preservation of the values of $B$ and $L$ under this change of variables is now straightforward to check.
\end{proof}

\begin{proof}[Proof of Theorem \ref{thm:existence_minimizer}]
    Recall that $\mathcal{A}$ is not empty.
    By assumption $\inf_\mathcal{A} B<\infty$.
    Take a minimizing sequence $\{\gamma_n\}_n\subset\mathcal{A}$ of $B$.
    Using the arclength reparametrization $\tilde{\gamma}_n$ to define $\bar{\gamma}_n(x):=\tilde{\gamma}_n(\tfrac{L_0(x-a)}{b-a})$, we obtain the admissible sequence $\{\bar{\gamma}_n\}\subset \mathcal{A}$ of constant-speed curves,
    \[
    |\bar{\gamma}_n'| = L_0,
    \]
    which is still minimizing, that is,
    \[
    \lim_{n\to\infty}B[\bar{\gamma}_n]= \inf_\mathcal{A} B < \infty.
    \]
    
    Now we observe that the sequence $\{\bar{\gamma}_n\}$ is bounded in $W^{2,2}(I;\R^n)$.
    Since
    \[
    \sup_n\|\bar{\gamma}_n'\|_\infty = L_0 < \infty,
    \]
    we also have 
    \[
    \sup_n\|\bar{\gamma}_n\|_\infty \leq \sup_n\left(|\bar{\gamma}_n(a)| + \int_a^b|\bar{\gamma}_n'|\right) \leq |P_0|+L_0<\infty.
    \]
    Thus the sequence is bounded in $W^{1,\infty}(I;\R^n)$ and in particular in $W^{1,2}(I;\R^n)$.
    In addition, since $\partial_s^2\bar{\gamma}_n=\frac{1}{L_0^2}\bar{\gamma}_n''$ and $ds=|\bar{\gamma}_n'|dx=L_0dx$, we have 
    \begin{align*}
        B[\bar{\gamma}_n] = \int_I|\partial_s^2\bar{\gamma}_n|^2ds = \frac{1}{L_0^3}\int_I|\bar{\gamma}_n''|^2dx,
    \end{align*}
    which means that
    \[
    \sup_n\|\bar{\gamma}_n''\|_2^2 = L_0^3 \sup_nB[\bar{\gamma}_n]<\infty.
    \]
    Hence the sequence $\{\bar{\gamma}_n\}$ is bounded in $W^{2,2}(I;\R^n)$.
    
    By compact embedding $W^{2,2}\subset\subset C^1$ and weak compactness of $W^{2,2}$, there exist a subsequence of $\{\bar{\gamma}_n\}$ (without relabelling) and a limit curve $\bar{\gamma}\in W^{2,2}(I;\R^n)$ such that $\bar{\gamma}_n\to\bar{\gamma}$ in $C^1(\bar{I};\R^n)$ and weakly in $W^{2,2}(I;\R^n)$ as $n\to\infty$.

    We finally confirm that $\bar{\gamma}$ is a minimizer.
    By $C^1$-convergence, the limit $\bar{\gamma}$ still satisfies the same boundary condition and the constant-speed property $|\bar{\gamma}'|=L_0$, so that $\bar{\gamma}\in \mathcal{A}$.
    In addition, by weak lower semicontinuity of the $L^2$-norm,
    \[
    \liminf_{n\to\infty} B[\bar{\gamma}_n] = \frac{1}{L_0^3}\liminf_{n\to\infty}\|\bar{\gamma}_n''\|_2^2 \geq \frac{1}{L_0^3}\|\bar{\gamma}''\|_2^2 = B[\bar{\gamma}].
    \]
    Therefore, $B[\bar{\gamma}]\leq \liminf_{n\to\infty} B[\bar{\gamma}_n] =\inf_\mathcal{A} B$.
\end{proof}

The almost same proof works for various other boundary conditions.
Typical examples include the \emph{pinned boundary condition}:
\[
    \mathcal{A}':=\left.\left\{\gamma\in W^{2,2}_\mathrm{imm}(I;\R^n)\,\,\right|
        L[\gamma]=L_0,\ \gamma(a)=P_0,\ \gamma(b)=P_1
    \right\} .
\]
In fact, as the first-order boundary condition was not used at any step of the proof of Theorem \ref{thm:existence_minimizer}, we can prove the following theorem just by replacing $\mathcal{A}$ with $\mathcal{A}'$.

\begin{theorem}[Existence of pinned minimizers]\label{thm:existence_minimizer_pinned}
    Let $L_0>0$, $P_0,P_1\in\R^n$ such that $|P_0-P_1|< L_0$.
    Then there exists $\bar{\gamma}\in\mathcal{A}'$ such that
    \[
    B[\bar{\gamma}] = \inf_{\gamma\in\mathcal{A}'} B[\gamma].
    \]
\end{theorem}

\subsection{Regularity}

Now we turn to the regularity issue.
In classical elastica theory, we can always prove that every minimizer, or more generally every critical point in the Sobolev framework must be smooth.

\begin{theorem}[Regularity of critical points]\label{thm:regularity_critical}
    Let $\gamma\in W^{2,2}_\textup{imm}(I;\R^n)$ satisfy the property that there exists $\lambda\in\R$ such that for any $\eta\in C^\infty_c(I)$,
    \[
    \frac{d}{d\varepsilon}\big(B[\gamma+\varepsilon\eta]+\lambda L[\gamma+\varepsilon\eta]\big)\Big|_{\varepsilon=0}=0.
    \]
    Then the arclength parametrization $\tilde{\gamma}$ of $\gamma$ is of class $C^\infty$, and hence a $\lambda$-elastica in the classical sense.
\end{theorem}

\begin{proof}
    First variation formulae \eqref{eq:FVLgeneral} and \eqref{eq:FVBgeneral} can also be verified in the Sobolev framework (via standard theory for interchanging differentiation and integration).
    Then the change of variables using $\sigma$ in Lemma \ref{lem:arclength_reparametrization_Sobolev} implies \eqref{eq:FVLunit} and \eqref{eq:FVBunit}.

    Using equations \eqref{eq:FVLunit} and \eqref{eq:FVBunit} for all $\eta\in C_c^\infty(J;\mathbf{R}^n)$, we find that 
    \[
    (2\tilde{\gamma}'')''+(3|\tilde{\gamma}''|^2\tilde{\gamma}'-\lambda\tilde{\gamma}')'=0
    \]
    in the distributional sense.
    Hence there is a constant vector $C\in\mathbf{R}^n$ such that
    \begin{equation}\label{eq:ODEdistribution}
      (2\tilde{\gamma}'')'+3|\tilde{\gamma}''|^2\tilde{\gamma}'-\lambda\tilde{\gamma}' -C =0
    \end{equation}
    holds in the distributional sense (see also Remark \ref{rem:distribution}).
    Since $\tilde{\gamma}\in W^{2,2}(J;\mathbf{R}^n)$, we have $|\tilde{\gamma}''|^2\in L^1(J)$ and $\tilde{\gamma}'\in L^\infty(J;\R^n)$ so that
    \[
    3|\tilde{\gamma}''|^2\tilde{\gamma}'-\lambda\tilde{\gamma}'-C\in L^1(J;\mathbf{R}^n).
    \] 
    Hence, equation \eqref{eq:ODEdistribution} implies that $\tilde{\gamma}'''\in L^1(J;\mathbf{R}^n)$, so that
    \[\tilde{\gamma}\in W^{3,1}(J;\mathbf{R}^n)\subset W^{2,\infty}(J;\mathbf{R}^n).\]
    (Now we can understand equation \eqref{eq:ODEdistribution} as weak derivatives.)
    Then
    \[
    3|\tilde{\gamma}''|^2\tilde{\gamma}'-\lambda\tilde{\gamma}'-C\in L^\infty(J;\mathbf{R}^n),
    \]
    and hence again by \eqref{eq:ODEdistribution} we have $\tilde{\gamma}\in W^{3,\infty}(J;\mathbf{R}^n)$.
    In general, if $\tilde{\gamma}\in W^{m,\infty}(J;\mathbf{R}^n)$ with $m\geq 3$, then the lower order terms $3|\tilde{\gamma}''|^2\tilde{\gamma}'-\lambda\tilde{\gamma}'-C$ can be (weakly) differentiated $m-2$ times and the resulting function is still bounded, and hence by \eqref{eq:ODEdistribution} the same holds for $2\tilde{\gamma}'''$, that is, $\tilde{\gamma}\in W^{m+1,\infty}(J;\mathbf{R}^n)$.
    This implies that $\tilde{\gamma}\in W^{m,\infty}(J;\mathbf{R}^n)$ for all $m\geq1$, and hence $\tilde{\gamma}\in C^\infty(\bar{J};\mathbf{R}^n)$.
\end{proof}

\begin{remark}\label{rem:distribution}
    One may circumvent distribution theory just by deducing equation \eqref{eq:ODEdistribution} (in the weak derivative sense) directly from the following fact:
    If $u,v\in L^p(I)$ satisfy $\int_I u\phi''=\int_I v\phi'$ for all $\phi\in C_c^\infty(I)$, then $u\in W^{1,p}(I)$ and $u'=-v+C$ holds for some $C\in\R$.
    We give an argument for the reader's convenience.
    Fix any $\psi_0\in C^\infty_c(I)$ such that $\int_I\psi_0=1$.
    For any $\tilde{\phi}\in C^\infty_c(I)$, if we define 
    \[
    \phi(x):=\int_a^x\left( \Big( \int_I\tilde{\phi} \Big)\psi_0-\tilde{\phi}\right), 
    \]
    then $\phi\in C^\infty_c(I)$, and hence by assumption $\int_I u\phi''=\int_I v\phi'$ we have
    \[
    -\int_I u\tilde{\phi}' = \int_I \left( -\Big( \int_I u\psi_0' \Big) + \Big( \int_I v\psi_0 \Big) - v \right)\tilde{\phi}.
    \]
    Letting $C:=-( \int_I u\psi_0' ) + ( \int_I v\psi_0 )$ completes the proof.
\end{remark}

\subsection{Method of Lagrange multipliers}

Throughout this section we focus on the regularity of (local) minimizers, just for simplicity of the presentation.

\begin{theorem}[Regularity of clamped minimizers]\label{thm:regularity_minimizer}
    Suppose the same assumption as in Theorem \ref{thm:existence_minimizer}.
    Let $\gamma$ be a local minimizer of $B$ in $\mathcal{A}$.
    Then the arclength parametrization $\tilde{\gamma}$ of $\gamma$ is of class $C^\infty$, and satisfies the elastica equation in the classical sense.
\end{theorem}

To prove this we recall the well-known method of Lagrange multipliers.

\begin{theorem}[Lagrange multiplier]\label{thm:multiplier}
    Let $X$ be a (real) Banach space and $U\subset X$ be an open set.
    Let $f,g_1,\dots,g_m\in C^1(U;\R)$, where $m\geq1$.
    If $x_0\in U$ is a local minimizer of $f$ on the set
    \[
    M:=\{x\in U \mid g_1(x)=\dots=g_m(x)=0 \},
    \]
    then there exists a nonzero vector $(\lambda_0,\lambda_1,\dots,\lambda_m)\in\R^{m+1}\setminus\{0\}$ such that
    \[
    \lambda_0 Df(x_0) + \sum_{j=1}^m \lambda_j Dg_j(x_0)=0.
    \]
    In addition, if $Dg_1(x_0),\dots,Dg_m(x_0)$ are linearly independent in the dual space $X^*$, then we can take $\lambda_0=1$, and in this case $\lambda_1,\dots,\lambda_m$ are uniquely determined.
\end{theorem}

The proof can be found in \cite[p.142]{Ward2001} for example.
An analogous property is also true for general critical points \cite[Prop 43.21]{Zeid3}.

To apply this method to our elastica problem, we confirm the following fact.

\begin{lemma}\label{lem:Frechet_B_L}
    Let $X\subset W^{2,2}(I;\R^n)$ be a closed subspace and $U\subset X\cap W^{2,2}_\mathrm{imm}(I;\R^n)$ be an open subset of $X$.
    Then the bending energy $B:\gamma\mapsto\int_I|\kappa|^2ds$ and the length functional $L:\gamma\mapsto\int_Ids$ are elements of $C^1(U;\R)$.
    In addition, their Fr\'{e}chet derivatives at $\gamma\in U$ are given by, for $\eta\in X$,
    \begin{align*}
        DB[\gamma](\eta) &= \int_I \left( 2\langle \gamma_{ss},\eta_{ss} \rangle - 3|\gamma_{ss}|^2\langle \gamma_s,\eta_s \rangle \right) ds,\\
        DL[\gamma](\eta) &= \int_I \langle \gamma_s,\eta_s \rangle ds.
    \end{align*}
\end{lemma}

\begin{proof}
    The G\^ateaux derivatives of $B$ and $L$ are given by the above formulae, or more precisely as in \eqref{eq:FVBgeneral} and \eqref{eq:FVLgeneral}, as already indicated in the proof of Theorem \ref{thm:regularity_critical}.
    It is straightforward to verify that the integrals in \eqref{eq:FVBgeneral} and \eqref{eq:FVLgeneral} are continuous in the $W^{2,2}$-topology as bounded linear operators; for example, letting $T_\gamma:\eta\mapsto \frac{d}{d\varepsilon}B[\gamma+\varepsilon\eta]|_{\varepsilon=0}$, then $T_\gamma\in \mathcal{B}(X,\R)$ and $\|T_{\gamma_n}-T_{\gamma}\|_{\mathcal{B}(X,\R)}\to0$ as $\|\gamma_n-\gamma\|_{W^{2,2}}\to0$.
    Therefore, those G\^ateaux derivatives are also Fr\'{e}chet derivatives, and the functionals $B$ and $L$ are of class $C^1$.
\end{proof}

\begin{proof}[Proof of Theorem \ref{thm:regularity_minimizer}]
    Let $\gamma\in\mathcal{A}$ be a local minimizer.
    Let $X:=W^{2,2}_0(I;\R^n)$, the closure of $C_c^\infty(I;\R^n)$ in $W^{2,2}(I;\R^n)$.
    Recall that any $\eta\in X$ satisfies $\eta=\eta'=0$ at the endpoints $\partial I$.
    Then the functions $f(\eta):=B[\gamma+\eta]$ and $g(\eta):=L[\gamma+\eta]-L_0$ are well-defined and of class $C^1$ for any $\eta\in U$ in a small neighborhood $U$ of $0$ in $X$.
    In addition, since $\gamma+\eta\in\mathcal{A}$ for $\eta\in X$, the origin $0\in X$ is a local minimizer of $f$ in $X\cap\{g=0\}$.
    Note also that $Dg(0)=DL[\gamma]\neq0$ since $\gamma$ is not a segment and $DL[\gamma]$ is given as in Lemma \ref{lem:Frechet_B_L}.
    Hence, by Theorem \ref{thm:multiplier} with $m=1$, there is $\lambda\in\R$ such that $Df(0)+\lambda Dg(0)=0$; by Lemma \ref{lem:Frechet_B_L}, for any $\eta\in X$,
    \[
    DB[\gamma](\eta)+\lambda DL[\gamma](\eta)=0.
    \]
    This in particular implies that $\gamma$ satisfies the assumption of Theorem \ref{thm:regularity_critical}, and hence has the desired regularity.
\end{proof}

The same argument also works for the pinned boundary condition.
In this case, due to the extra freedom at the endpoints, we further deduce an additional boundary condition on the curvature.

\begin{theorem}[Regularity of pinned minimizers]\label{thm:regularity_minimizer_pinned}
    Suppose the same assumption as in Theorem \ref{thm:existence_minimizer_pinned}.
    Let $\gamma$ be a local minimizer of $B$ in $\mathcal{A}'$.
    Then the arclength parametrization $\tilde{\gamma}$ of $\gamma$ is of class $C^\infty$, and satisfies the elastica equation in the classical sense.
    In addition, $\kappa=0$ holds at the endpoints.
\end{theorem}

\begin{proof}
    Arguing similarly to Theorem \ref{thm:regularity_minimizer} by replacing $X=W^{2,2}_0(I;\R^n)$ with $X':=W^{2,2}(I;\R^n)\cap W^{1,2}_0(I;\R^n)$, we find that for any $\eta\in X'$,
    \[
    DB[\gamma]+\lambda DL[\gamma]=0.
    \]
    Hence again the desired regularity holds.
    Not only that, in view of \eqref{eq:FVLintegralbyparts_BC}, \eqref{eq:FVBintegralbyparts_BC}, and \eqref{eq:elastica_ODE_fourthorder}, integration by parts yields
    \[
    [2\langle \gamma_{ss},\eta_s \rangle]_a^b=0
    \]
    for all $\eta\in X$.
    (Note that we always have $\eta(a)=\eta(b)=0$.)
    Since $\eta_s(a)$ and $\eta_s(b)$ may be arbitrary, we deduce that $\gamma_{ss}(a)=\gamma_{ss}(b)=0$.
\end{proof}

\section{Li--Yau type inequality and related problems}\label{sec:Li-Yau}

\subsection{Li--Yau type inequality}

Armed with the contents of the previous sections, we are now ready to prove the most important step towards Theorem \ref{thm:intro_Li-Yau}.

\begin{theorem}[Unique minimality of the leaf \cite{Miura_LiYau}]\label{thm:leaf_minimality}
    Let $\gamma\in W^{2,2}_\mathrm{imm}(I;\R^n)$ satisfy $\gamma(a)=\gamma(b)$.
    Let $\gamma^*$ be the half-fold figure-eight elastica (leaf) defined by $\gamma^*(s):=\gamma_w(s-K(m^*),m^*)$ for $s\in[0,2K(m^*)]$.
    Then 
    \[
    \bar{B}[\gamma]\geq \bar{B}[\gamma^*],
    \]
    where equality holds if and only if $\gamma$ coincides with $\gamma^*$ up to similarity and reparameterization.
\end{theorem}

\begin{proof}
    Without loss of generality we may temporarily assume that $\gamma(a)=\gamma(b)=0$ and $L[\gamma]=1$.
    By Theorem \ref{thm:regularity_minimizer_pinned} there is a minimizer $\bar{\gamma}$ of $B$ among all such curves, and $\bar{\gamma}$ is a smooth elastica with vanishing curvature at the endpoints.
    By Corollary \ref{cor:uniqueness_dim_rigidity} the elastica $\bar{\gamma}$ must be planar, so we may assume $n=2$.
    Hence $\bar{\gamma}$ falls into the classification in Theorem \ref{thm:planar_explicit}.
    Since $\bar{\gamma}$ cannot be a segment, the only possible case (allowing vanishing curvature) is Case II (wavelike elastica).
    Then, using the vanishing-curvature boundary condition as well as periodicity and symmetry, we deduce that there are $m\in(0,1)$ and a positive integer $N$ such that, up to similarity and reparametrization,
    \[
    \bar{\gamma}(s)=\gamma_w(s-K(m),m)
    \]
    for $s\in[0,2K(m)N]$.
    Directly computing the energy $\bar{B}$ and using periodicity, we deduce from energy minimality that $N=1$.
    In addition, Lemma \ref{lem:elliptic_integral_wavelike} with the fact that $0=\bar{\gamma}(2K(m))-\bar{\gamma}(0)=2(2E(m)-K(m))e_1$ implies that $m=m^*$.
\end{proof}

A direct computation (as in \cite[Lemma 2.5]{Miura_LiYau}) ensures that the energy of the leaf can be computed as
\[
    \bar{B}[\gamma^*]=\varpi^* \big( =32(2m^*-1)E(m^*) \big).
\]
With this fact, Theorem \ref{thm:leaf_minimality} is exactly same as \cite[Proposition 2.6]{Miura_LiYau}.
Once this is established, Theorem \ref{thm:intro_Li-Yau} follows by cutting the curve at the point of multiplicity and apply Theorem \ref{thm:leaf_minimality} to each piece.
The fact that minimizers have equal-length leaves follows by an easy application of the AM-HM inequality.
The non-optimality part is still delicate and depends on Andr\'e's theorem \cite{Andre1996}.
For more details, see the original paper \cite{Miura_LiYau}.

In what follows, we discuss some related topics and open problems.

We begin with a directly related open problem about the non-optimal case of Theorem \ref{thm:intro_Li-Yau}.

\begin{problem}\label{prob:odd_planar}
    What is the minimizer of the normalized bending energy $\bar{B}$ among closed planar curves $\gamma\in W^{2,2}_\textup{imm}(\R/\Z;\R^2)$ with a point of odd multiplicity $r\geq3$?
\end{problem}

\subsection{Elastic knot}

Elastic knot theory is studied as a more accurate mathematical model of elastic wires.
One mathematical formulation (via singular perturbation) is recently given by a remarkable work of Gerlach--Reiter--von der Mosel \cite{Gerlach2017}.
In their paper it is shown that the doubly-covered circle is an appropriate ``minimizer'' in the trefoil class, or even in a more general family of torus knot classes.

One of the most important open problems in this field would be the following conjecture, which was first stated by Gallotti and Pierre-Louis in 2007 \cite{Gallotti_PierreLouis_2007} and then mathematically formulated in \cite{Gerlach2017}.

\begin{conjecture}[Circular elastic knots {\cite[Conjecture 7.1]{Gerlach2017}}]
    The $N$-fold circle is the unique elastic knot (in the sense of \cite{Gerlach2017}) for any knot class whose braid index and bridge index coincide with $N$.
\end{conjecture}

On the other hand, non-circular elastic knots can also be realized experimentally.
For example, motivated by some previous works, in \cite{MMR23} we have posed the following

\begin{conjecture}[Elastic teardrop-heart \cite{MMR23}]
    The elastic teardrop-heart (obtained in \cite{MMR23}) is the unique elastic knot (in the sense of \cite{Gerlach2017}) for the figure-eight knot class.
\end{conjecture}

In addition, we expect that leafed elasticae (attaining equality of Theorem \ref{thm:intro_Li-Yau}) provide a new family of candidates of stable elastic knots.
In fact, the elastic propeller can be made by deforming a circular (unknotted) wire as in Figure \ref{fig:propeller_photo}.

In \cite{Miura_LiYau} the author conjectured that the elastic propeller is the stable elastic unknot of second smallest energy (the least energy one is clearly the circle).
To give a detailed statement, we need not use singular perturbation as in \cite{Gerlach2017} but only need to introduce a suitable closure of a knot class.
For a given (tame) knot class $K$, we define 
\[
\mathcal{K}:=\overline{\{\gamma\in W^{2,2}(\R/\Z;\R^3) \mid |\gamma'|\equiv1,\ \gamma\in K \}},
\]
the $W^{2,2}$-closure of the class of arclength parametrized closed curves which have unit length and belong to the knot class $K$.
In particular, taking the closure allows us to deal with self-intersecting curves.

In this paper we define that $\gamma\in\mathcal{K}$ is a \emph{stable elastic knot (for the knot class $K$)} if there are a neighborhood $U$ of $\gamma$ in $\mathcal{K}$ and a connected component $U'$ of $U\cap K$ such that $\gamma\in \overline{U'}$ and $B[\gamma]=\min_{\overline{U'}}B$.
We also define that $\gamma\in\mathcal{K}$ is a \emph{strictly} stable elastic knot if in addition $B[\xi]>B[\gamma]$ holds for any $\xi\in\overline{U'}$ that is not isometric to $\gamma$.
Recall that by Langer--Singer's theorem, any stable elastic knot other than the circle must have a self-intersection (or equivalently, lie in the boundary of the class $\mathcal{K}$).

\begin{conjecture}[Elastic propeller \cite{Miura_LiYau}]
    The elastic propeller is a strictly stable elastic knot for the unknot class.
    In addition, among all stable elastic knots for the unknot class, the elastic propeller has the second smallest energy.
\end{conjecture}

In fact, we expect that the family of leafed elasticae provides more candidates of stable elastic knots.
This conjecture is partly motivated by the fact that a $4$-leafed elastica can also be realized experimentally; this was told by Ben Andrews to the author in 2021.

\begin{conjecture}
    There are integers $r\geq4$ such that a closed $r$-leafed elastica is a stable elastic knot for the unknot class.
\end{conjecture}

\begin{problem}
    Is the above family of integers $r$ infinite?
\end{problem}

Note that the above problems do not say anything about uniqueness (or strictness).
In order to obtain physical uniqueness for four or more leaves, we would also need to take account of higher-order effects as in the singular perturbation theory of \cite{Gerlach2017}, due to non-uniqueness of leafed elasticae explained in \cite[Remark 3.18]{Miura_LiYau}.

\subsection{$p$-Elastica}

The natural $L^p$-counterpart of elastica theory is called $p$-elastica theory.
More precisely, for $p\in(1,\infty)$, a \emph{$p$-elastica} $\gamma\in W^{2,p}_\textup{imm}(I;\R^n)$ is defined as a critical point of the \emph{$p$-bending energy}
\[
B_p[\gamma]:=\int_I|\kappa|^pds
\]
among fixed-length curves, or equivalently for some $\lambda\in\R$,
\[
\nabla B_p[\gamma] +\lambda\nabla L[\gamma]=0.
\]
In general, this Euler--Lagrange equation takes the form of a second order ODE for the curvature whose top-order term is singular ($p<2$) or degenerate ($p>2$).
This leads to loss of regularity for solutions and makes their classification much more complicated.
In particular, Watanabe \cite{nabe14} found that for $p>2$ there is a qualitatively novel type of critical point called \emph{flat-core}.

Classification theory for planar $p$-elasticae is developed in a recent series of works by the author and Yoshizawa
\cite{MYarXiv2203,MYarXiv2209,MY_2024_Crelle,MYarXiv2310}.
In particular, this theory implies

\begin{theorem}[Classification of closed planar $p$-elasticae {\cite[Theorem 5.6]{MYarXiv2203}}]
    Let $p\in(1,\infty)$.
    Then any closed planar $p$-elastica is given by either a circle or a figure-eight $p$-elastica, possibly multiply covered.
\end{theorem}

The shapes of figure-eight $p$-elasticae depend on the values of $p$ as in Figure \ref{fig:p-figure8}.

\begin{figure}[htbp]
  \includegraphics[scale=0.2]{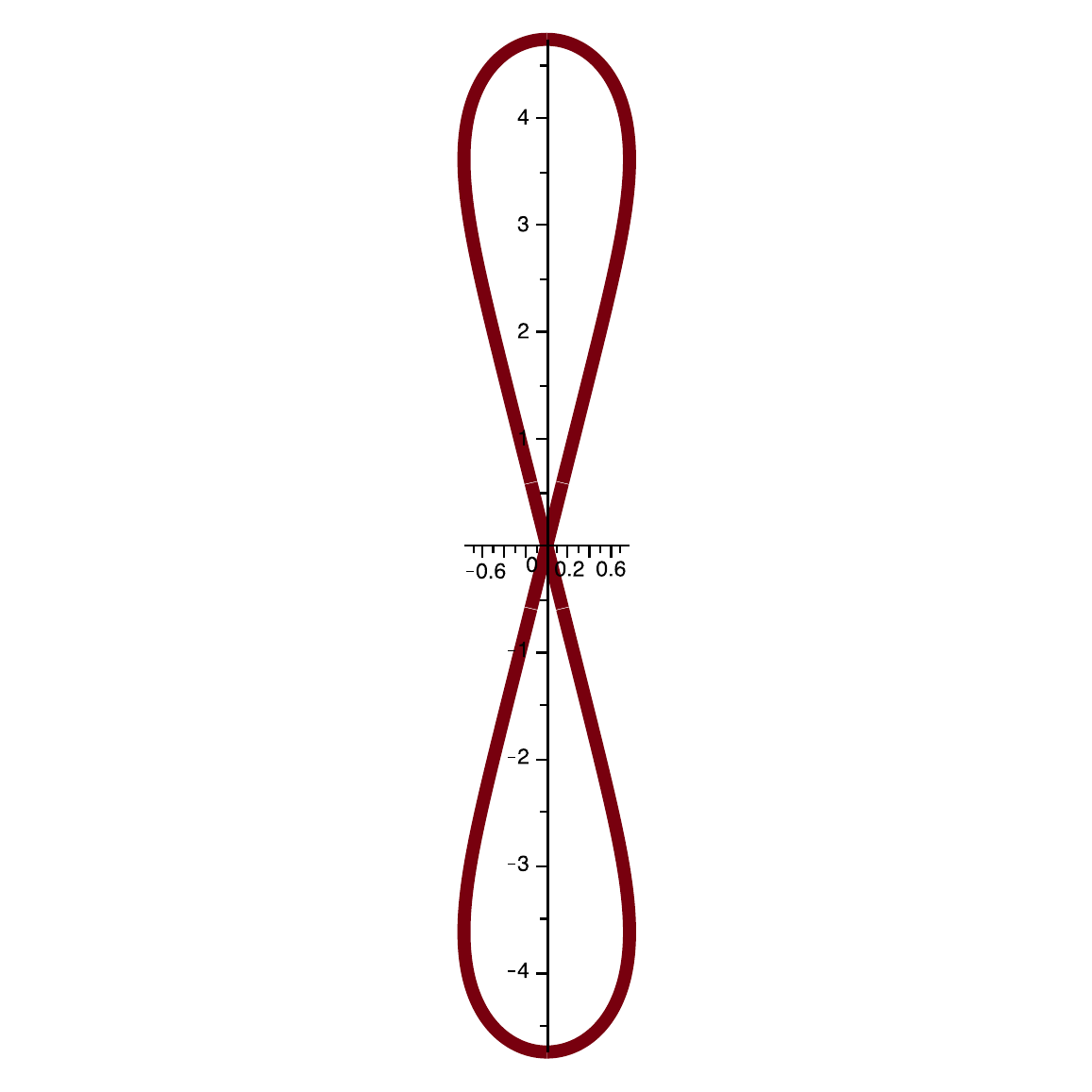}
  \includegraphics[scale=0.2]{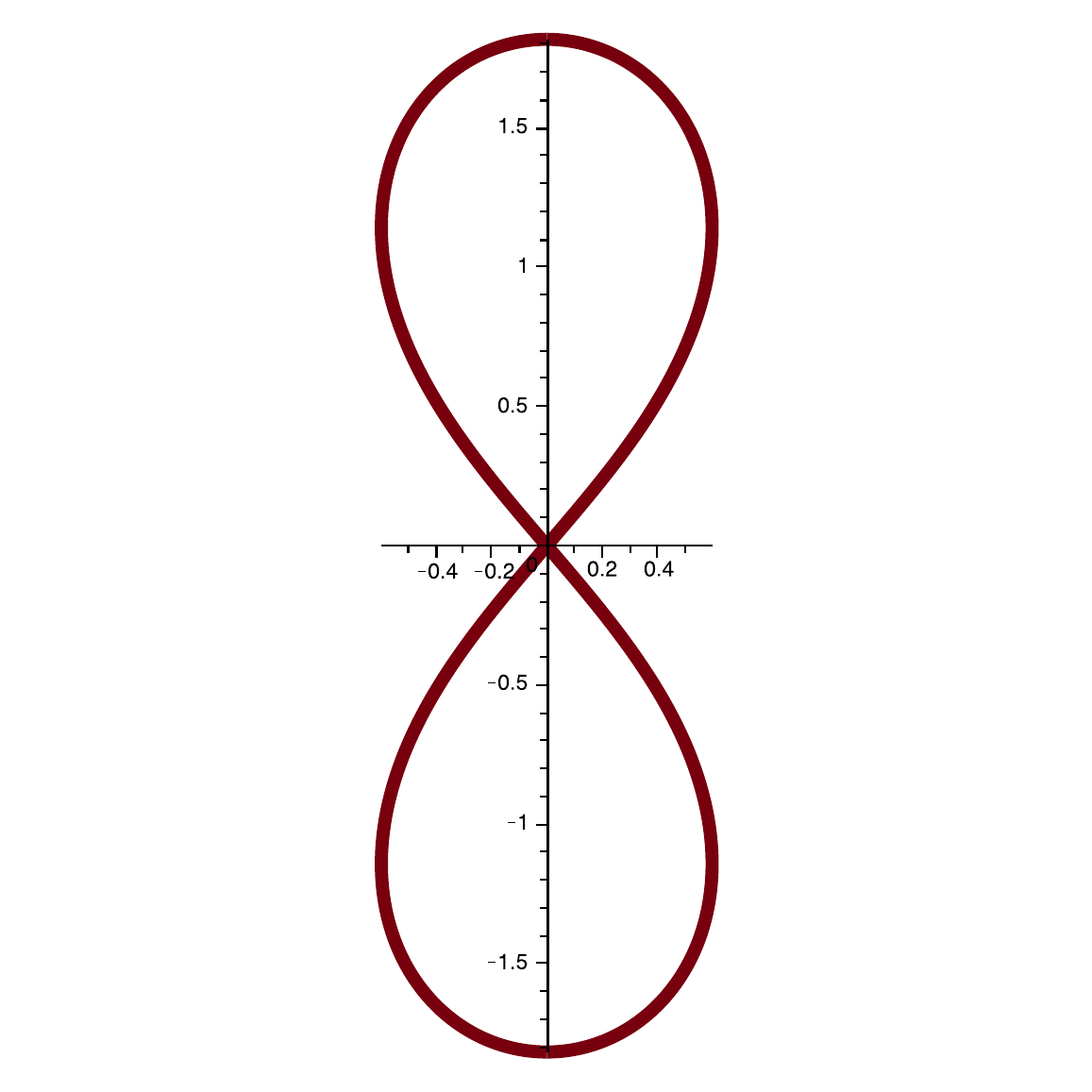}
  \includegraphics[scale=0.2]{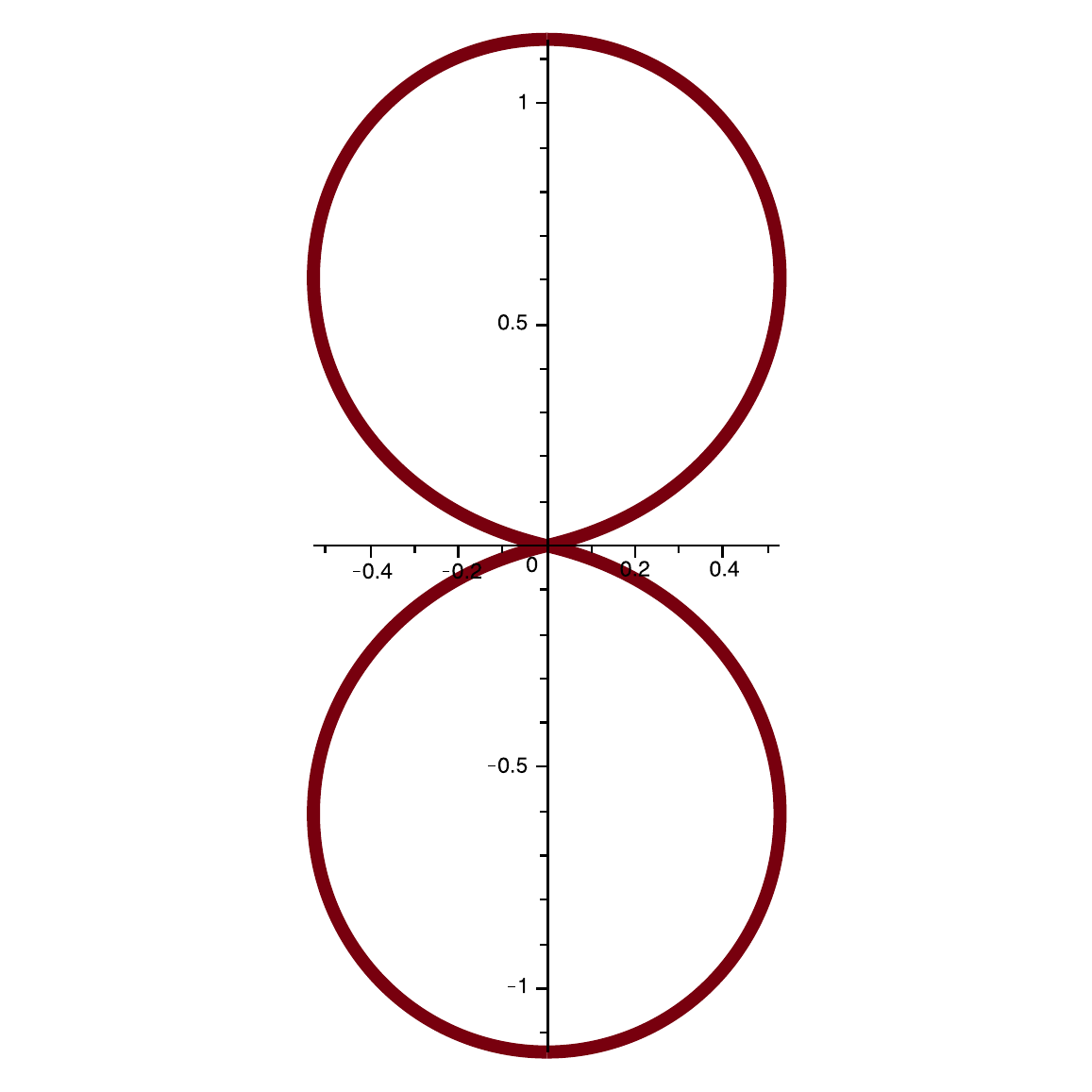}
\caption{Figure-eight $p$-elasticae with $p=\frac{6}{5}$, $p=2$, $p=10$ (from left to right) \cite{MYarXiv2203}.}
\label{fig:p-figure8}
\end{figure}

This result motivates us to extend the Li--Yau type inequality to the normalized $p$-bending energy
\[
\bar{B}_p[\gamma] := L[\gamma]^{p-1}B_p[\gamma].
\]
In \cite{MYarXiv2209} we solved the pinned boundary value problem for planar $p$-elasticae, and thus succeeded in generalizing the Li--Yau type inequality in the planar case.
Let $\varpi_p^*>0$ denote the energy $\bar{B}_p$ of the half-fold figure-eight $p$-elastica.

\begin{theorem}[{\cite{MYarXiv2209}}]\label{thm:p-Li-Yau}
    Let $p\in(1,\infty)$, and $r\geq2$ be an integer.
    Let $\gamma:\R/\Z\to\R^2$ be an immersed closed curve of class $W^{2,p}$ with a point of multiplicity $r$.
    Then
    \[
    \bar{B}_p[\gamma] \geq \varpi_p^* r^p.
    \]
    In addition, equality holds if and only if $\gamma$ is a closed $r$-leafed $p$-elastica.
\end{theorem}

As in the case $p=2$, the existence of closed planar leafed $p$-elasticae is a delicate problem, and in fact sensitively depends on the value of $p\in(1,\infty)$.
However, in \cite{MYarXiv2209} we proved monotonicity of the crossing angle with respect to $p$, and thus found a unique exponent $p$ such that equality holds for every multiplicity $r\geq2$.

\begin{theorem}[{\cite{MYarXiv2209}}]\label{thm:unique_exponent}
    There is a unique exponent $p\in(1,\infty)$ such that for every integer $r\geq2$ there exists a closed planar $r$-leafed $p$-elastica.
\end{theorem}

The reason why this theorem holds is that for the above unique exponent the figure-eight $p$-elastica has crossing angle $\frac{2\pi}{3}$, so that three leaves can be closed as in Figure \ref{fig:3-leaf}.
The value of this exponent is numerically computed as $p\simeq1.5728$.

\begin{center}
    \begin{figure}[htbp]
      \includegraphics[scale=0.2]{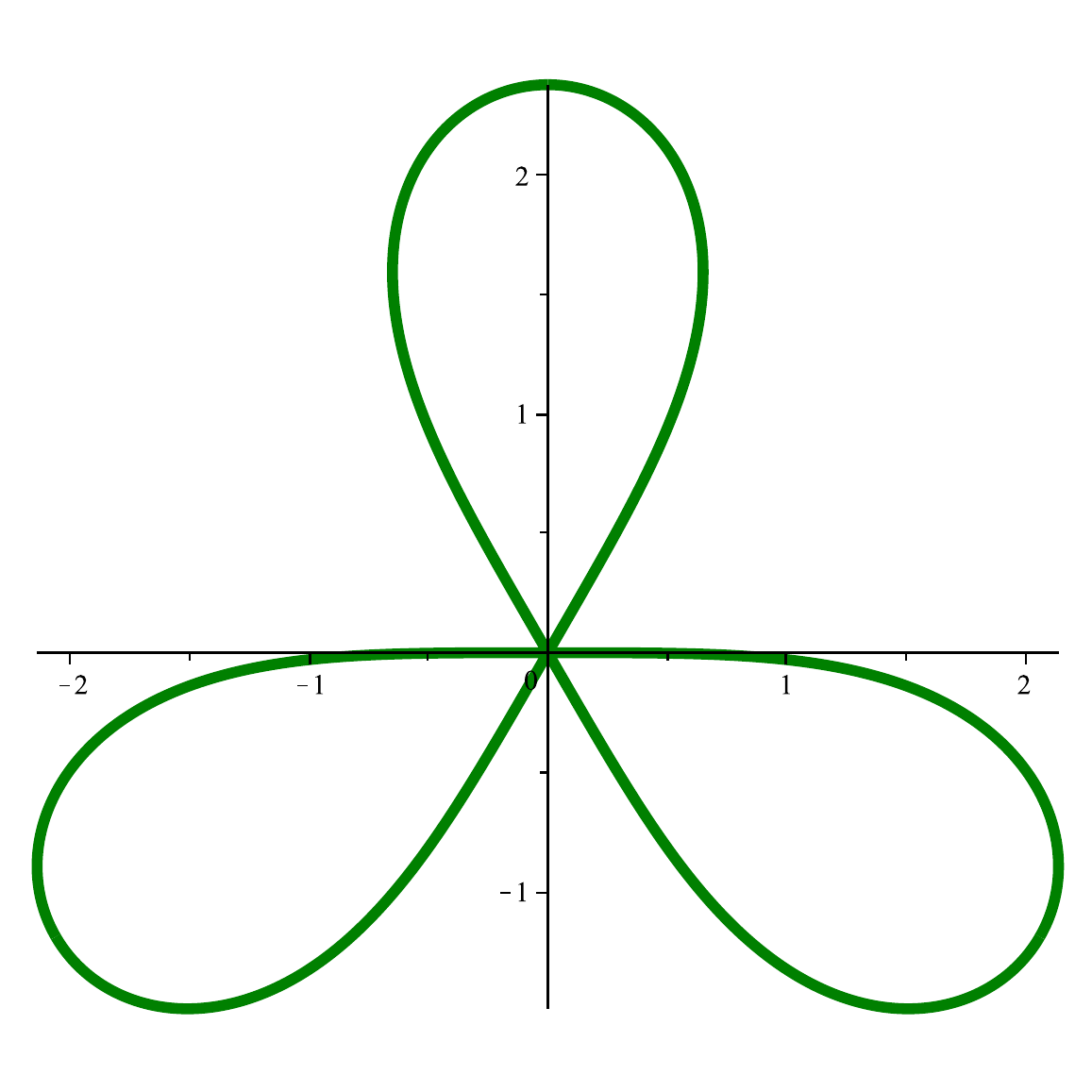}
  \caption{The unique closed planar $3$-leafed $p$-elastica \cite{MYarXiv2209}.}
  \label{fig:3-leaf}
  \end{figure}
\end{center}

Now we discuss several open problems.
A straightforward problem is to extend Theorem \ref{thm:p-Li-Yau} from $\R^2$ to $\R^n$.
This is work in progress and will be addressed in our future work \cite{GruenMiura}.
It is also natural to ask a generalized version of Problem \ref{prob:odd_planar} about optimal curves which are not leafed $p$-elasticae.
See also \cite{MYarXiv2209} for open problems on $p$-elastic networks.

Here we mention other kinds of (vague) problems.

\begin{problem}
    Is there a simple (algebraic) characterization of the unique exponent $p$ in Theorem \ref{thm:unique_exponent}?
\end{problem}

\begin{problem}
    Is there a simple (algebraic) characterization of pairs $(p,r)$ for which there exists a closed planar $r$-leafed $p$-elastica?
\end{problem}

\subsection{Elastic flow}

The $L^2$-gradient flow of the energy $B+\lambda L$ for $\lambda>0$ is called the \emph{(length-penalized) elastic flow}.
This flow is parabolic but of higher order so that the flow loses many kinds of positivity due to the lack of maximum principles.
Embeddedness is a typical instance of positivity.
In \cite{MMR23} optimal energy thresholds for embeddedness-preservation of closed elastic flows in $\R^n$ are obtained.
This problem is closely related to the Li--Yau type inequality, since the inequality implies an optimal energy upper bound to ensure embeddedness.
In fact, the proof of the case $n\geq3$ directly relies on Theorem \ref{thm:intro_Li-Yau} (for $n=2$ the threshold value becomes higher and an additional argument is needed).

Recently, the \emph{$p$-elastic flow}, where $B+\lambda L$ is replaced by $B_p+\lambda L$, is also studied in \cite{OW23,BHV,BVH}.
The $p$-elastic flow for $p\neq2$ is considerably delicate in view of regularity, and in fact even the optimal convergence topology for stationary solutions is left open (as mentioned in \cite{MYarXiv2203}).

Concerning self-intersections, the following natural problem is also open.

\begin{problem}
    Let $p\in(1,\infty)$.
    Determine the optimal energy thresholds (as in \cite{MMR23}) such that $p$-elastic flows preserve embeddedness of initial closed curves.
\end{problem}

\appendix
\section{Jacobi elliptic integrals and functions}\label{sec:elliptic}

In this appendix we collect definitions and basic properties of Jacobi elliptic integrals and functions.
For more details see classical textbooks, e.g.\ \cite{Whittaker1962}.

We first give a general remark on the notation.
Each of elliptic integrals and functions is defined as a one-parameter family of functions, like $F(x,m)$ and $\sn(x,m)$.
Throughout this section we use the notation $m\in[0,1]$, which is usually called the \emph{parameter} in the literature.
However, we shall emphasize that the \emph{elliptic modulus} $k\in[0,1]$ defined through the relation $m=k^2$ is also frequently used.
Here we choose $m$ just to simplify computations.


\subsection{Jacobi elliptic integrals}

\begin{definition}[Elliptic integral of first kind]
  For a parameter $m\in[0,1)$ and $x\in\R$ we define the \emph{incomplete elliptic integral of first kind} by
  \begin{equation*}
    F(x,m):=\int_0^x\frac{d\theta}{\sqrt{1-m\sin^2\theta}},
  \end{equation*}
  and the \emph{complete elliptic integral of first kind} by
  \begin{align*}
    K(m) &:= F\left(\frac{\pi}{2},m\right) = \int_0^{\pi/2}\frac{d\theta}{\sqrt{1-m\sin^2\theta}}>0.
  \end{align*}
\end{definition}

\begin{definition}[Elliptic integral of second kind]
  For a parameter $m\in[0,1)$ and $x\in\R$ we define the \emph{incomplete elliptic integral of second kind} by
  \begin{equation*}
    E(x,m):=\int_0^x\sqrt{1-m\sin^2\theta}d\theta,
  \end{equation*}
  and the \emph{complete elliptic integral of second kind} by
  \begin{align*}
    E(m) &:= E\left(\frac{\pi}{2},m\right) = \int_0^{\pi/2}\sqrt{1-m\sin^2\theta}d\theta >0.
  \end{align*}
\end{definition}

The following properties are obtained from the definition.

\begin{proposition}
  For any $m\in[0,1)$ and $x\geq0$ (resp.\ $x\leq0$),
  \begin{equation*}
    F(x,m)\geq E(x,m) \quad \big( \text{resp.}\ F(x,m)\leq E(x,m) \big),
  \end{equation*}
  where equality holds if and only if either $x=0$ or $m=0$.
  In particular,
  \begin{equation*}
    K(m)\geq E(m),
  \end{equation*}
  where equality holds if and only if $m=0$.
\end{proposition}

\begin{proposition}[Dependence on variable]
  For each $m\in[0,1)$, both $F(\cdot,m)$ and $E(\cdot,m)$ are real analytic, strictly increasing, bijective functions from $\R$ to $\R$ such that
  \begin{align*}
    \partial_xF(x,m) &= \frac{1}{\sqrt{1-m\sin^2x}}\in[1,\frac{1}{\sqrt{1-m}}],\\
    \partial_xE(x,m) &= \sqrt{1-m\sin^2x}\in[\sqrt{1-m},1].
  \end{align*}
  In addition, we have the odd symmetry
  \begin{align*}
    F(-x,m) &= -F(x,m),\\
    E(-x,m) &= -E(x,m),
  \end{align*}
  and also the (arithmetic) quasi-periodicity,
  \begin{align*}
    F(x+\pi,m) &= F(x,m)+2K(m),\\
    E(x+\pi,m) &= E(x,m)+2E(m).
  \end{align*}
  In particular, for any integer $\ell\in\Z$,
  \begin{align*}
    F\left(\frac{\ell\pi}{2},m\right) &= \ell K(m),\\
    E\left(\frac{\ell\pi}{2},m\right) &= \ell E(m).
  \end{align*}
\end{proposition}

\begin{remark}
  Analyticity follows by the fact that an antiderivative of an analytic function is analytic,  also a composition of analytic functions is analytic.
\end{remark}

\begin{proposition}[Dependence on parameter: Incomplete case]
  For each positive $x>0$, both $F(x,\cdot)$ and $E(x,\cdot)$ are smooth functions on the interval $[0,1)$.
  In addition, $F(x,\cdot)$ is a strictly increasing function such that
  \begin{align*}
    F(x,0) &=x, \quad \lim_{m\to1}F(x,m)=\infty,\\
    \dm F(x,m)&>0, \quad \dm F(x,0)=\frac{1}{4}(2x-\sin{2x}).
  \end{align*}
  and $E(x,\cdot)$ is a strictly decreasing function such that
  \begin{align*}
    E(x,0)&=x, \quad \lim_{m\to1}E(x,m)=\int_0^x|\cos\theta|d\theta, \\
    \dm E(x,m)&<0, \quad \dm E(x,0)=-\frac{1}{4}(2x-\sin{2x}).
  \end{align*}
\end{proposition}

\begin{remark}
  The signs of $\partial_mF$ and $\partial_mE$ and their values at $m=0$ follow by the following general expressions:
  \begin{align*}
    \dm F(x,m) &= \int_0^x\frac{\sin^2\theta}{2(1-m\sin^2\theta)^{3/2}}d\theta,\\
    \dm E(x,m) &= -\int_0^x\frac{\sin^2\theta}{2(1-m\sin^2\theta)^{1/2}}d\theta.
  \end{align*}
\end{remark}

\begin{corollary}[Dependence on parameter: Complete case]
  $K:[0,1)\to\R$ is a smooth, strictly increasing function such that
  \begin{equation*}
    K(0)=\frac{\pi}{2}, \quad \lim_{m\to1}K(m)=\infty, \quad K'(m)>0, \quad K'(0)=\frac{\pi}{8},
  \end{equation*}
  and $E:[0,1)\to\R$ is a smooth, strictly decreasing function such that
  \begin{equation*}
    E(0)=\frac{\pi}{2}, \quad \lim_{m\to1}E(m)=1, \quad E'(m)<0, \quad E'(0)=-\frac{\pi}{8}.
  \end{equation*}
\end{corollary}

\begin{proposition}[Representation of parameter derivatives: Incomplete case]
  The following derivative formulae hold: For $m\in[0,1)$ and $x\in\R$,
  \begin{align*}
    \dm F(x,m) &= \frac{E(x,m)}{2m(1-m)}-\frac{F(x,m)}{2m} - \frac{\sin{x}\cos{x}}{2(1-m)\sqrt{1-m\sin^2{x}}},\\
    \dm E(x,m) &= \frac{E(x,m)-F(x,m)}{2m}.
  \end{align*}
\end{proposition}

\begin{remark}
  The first identity needs a nontrivial trick in addition to the above expression of $\dm F(x,m)$.
  More precisely, we use the identity
  \begin{equation*}
    \int_0^x\frac{d\theta}{(1-m\sin^2\theta)^{3/2}} = \frac{E(x,m)}{1-m} - \frac{m}{1-m} \frac{\sin{x}\cos{x}}{\sqrt{1-m\sin^2{x}}},
  \end{equation*}
  which follows since
  \begin{equation*}
    \frac{1}{(1-m\sin^2\theta)^{3/2}} = \frac{\sqrt{1-m\sin^2\theta}}{1-m} -\frac{m}{1-m} \dx \left( \frac{\sin{x}\cos{x}}{\sqrt{1-m\sin^2{x}}} \right).
  \end{equation*}
  The second identity can easily be obtained by the above expression of $\dm E(x,m)$.
\end{remark}

\begin{corollary}[Representation of parameter derivatives: Complete case]
  The following derivative formulae hold: For $m\in[0,1)$,
  \begin{align*}
    K'(m) &= \frac{E(m)-(1-m)K(m)}{2m(1-m)},\\
    E'(m) &= \frac{E(m)-K(m)}{2m}.
  \end{align*}
\end{corollary}

\begin{proposition}[Series expansion by parameter]
  For $m\in[0,1)$,
  \begin{align*}
    K(m) &= \frac{\pi}{2}\sum_{n=0}^\infty\left(\frac{(2n-1)!!}{(2n)!!}\right)^2m^n,\\
    E(m) &= \frac{\pi}{2}\sum_{n=0}^\infty\left(\frac{(2n-1)!!}{(2n)!!}\right)^2\frac{1}{1-2n}m^n,
  \end{align*}
  where $(-1)!!:=1$ and $0!!:=1$.
\end{proposition}

These expansions are used in \cite{Miura_LiYau} to estimate the numerical value of the parameter corresponding to the figure-eight elastica.

\begin{remark}
  We briefly demonstrate the computation.
  Since the $n$-th derivative of integrand $(1-m\sin^2\theta)^{-1/2}$ of $K$ is represented by $\frac{(2n-1)!!}{2^n}(\sin^{2n}\theta)(1-m\sin^2\theta)^{-\frac{2n-1}{2}}$, and hence by Wallis' formula $\int_0^{\pi/2}\sin^{2n}\theta d\theta=\frac{\pi}{2}\frac{(2n-1)!!}{(2n)!!}$ we have
  \begin{equation*}
    \frac{1}{n!}K^{(n)}(0)=\frac{(2n-1)!!}{(2n)!!}\int_0^{\pi/2}\sin^{2n}\theta d\theta = \frac{\pi}{2}\left(\frac{(2n-1)!!}{(2n)!!}\right)^2.
  \end{equation*}
  The case of $E$ can similarly be computed.
\end{remark}

\begin{remark}
  The hypergeometric function ${}_2F_1[a,b;c;z]$ is defined for $|z|<1$ and $c\not\in\Z_{\leq0}$ by the power series:
  \begin{equation*}
    {}_2F_1[a,b;c;z] = \sum_{n=0}^\infty \frac{(a)_n(b)_n}{(c)_n}\frac{z^n}{n!},
  \end{equation*}
  where $(q)_0:=1$ and otherwise $(q)_n:=q(q+1)\cdots(q+n-1)$.
  In terms of this function, we have
  \begin{align*}
    K(m) &= \frac{\pi}{2}{}_2F_1[\tfrac{1}{2},\tfrac{1}{2};1;m],\\
    E(m) &= \frac{\pi}{2}{}_2F_1[-\tfrac{1}{2},\tfrac{1}{2};1;m].
  \end{align*}
  The algebraic independence of the values of those functions over $\overline{\Q}$ follows by Andr\'e's theorem \cite{Andre1996}. 
\end{remark}

\subsection{Jacobi elliptic functions}

\begin{definition}[Jacobi amplitude]
  For $m\in[0,1)$ we define the \emph{(Jacobi) amplitude} $\am(\cdot,m):\R\to\R$ by the inverse function of $F(\cdot,m)$:
  $$\am(x,m):=F^{-1}(x,m).$$
\end{definition}

\begin{definition}[Jacobi elliptic functions]\label{def:B2}
  For $m\in [0,1)$ and $x\in\R$ we define the \emph{elliptic sine}, \emph{elliptic cosine}, and \emph{elliptic delta} by
  \begin{align*}
    \sn(x,m) & := \sin(\am(x,m)),\\
    \cn(x,m) & := \cos(\am(x,m)),\\
    \dn(x,m) & := \sqrt{1-m\sin^2(\am(x,m))},
  \end{align*}
  respectively.
  In addition, for $m=1$, we define
  \begin{align*}
    \sn(x,1) & := \tanh{x},\\
    \cn(x,1) & := \sech{x},\\
    \dn(x,1) & := \sech{x}.
  \end{align*}
  All they are called \emph{(Jacobi) elliptic functions}.
\end{definition}

\begin{remark}[Hyperbolic functions]
  For the reader's convenience, we also put down the definitions of some hyperbolic functions:
  \begin{align*}
    \sinh{x} &:= \frac{e^x-e^{-x}}{2},\\
    \cosh{x} &:= \frac{e^x+e^{-x}}{2},\\
    \tanh{x} &:= \frac{\sinh{x}}{\cosh{x}}=\frac{e^x-e^{-x}}{e^x+e^{-x}},\\
    \sech{x} &:= \frac{1}{\cosh{x}} = \frac{2}{e^x+e^{-x}}.
  \end{align*}
\end{remark}

\begin{remark}
  The definition in the case of $m=1$ is independently given but natural in the sense that all they are indeed limit functions as $m\to1$.
\end{remark}

\begin{proposition}[Trigonometric identity]
  For $m\in[0,1]$ and $x\in\R$,
  \begin{align*}
    \cn^2(x,m) + \sn^2(x,m) & = 1, \\
    \dn^2(x,m) + m \sn^2(x,m) &= 1,\\
    \dn^2(x,m) - m \cn^2(x,m) &= 1-m.
  \end{align*}
\end{proposition}

\begin{proposition}[Basic properties of amplitude]
  For each $m\in[0,1)$, the amplitude function is a real analytic, strictly increasing, bijective function from $\R$ to $\R$ such that
  $$\dx\am(x,m)=\dn(x,m).$$
  In addition, we have the odd symmetry and (arithmetic) quasi-periodicity:
  \begin{align*}
    \am(-x,m) &= \am(x,m),\\
    \am(x+2K(m),m) &= \am(x,m)+\pi.
  \end{align*}
  In particular, for any integer $\ell\in\Z$,
  \begin{align*}
    \am\big(\ell K(m),m \big) &= \frac{\ell\pi}{2}.
  \end{align*}
\end{proposition}

\begin{remark}
  The analyticity of the amplitude follows since the inverse of an analytic function is analytic.
\end{remark}

\begin{remark}
  By $\dx E(x,m)=\sqrt{1-m\sin^2{x}}$ and $\dx\am(x,m)=\dn(x,m)$ we deduce that
  \begin{equation*}
    \dx \big( E(\am(x,m),m) \big) = \dx E(\am(x,m),m) \dx\am(x,m) = \dn^2(x,m).
  \end{equation*}
\end{remark}

\begin{proposition}[Basic properties of elliptic functions]
  For each parameter $m\in[0,1]$, the elliptic functions $\sn(\cdot,m)$, $\cn(\cdot,m)$, and $\dn(\cdot,m)$ are real analytic, and the following derivative formulae hold:
  \begin{align*}
    \dx \sn(x,m) & =  \cn(x,m) \dn(x,m) , \\
    \dx \cn(x,m) & = - \sn(x,m) \dn(x,m), \\
    \dx \dn(x,m) & = - m \cn(x,m) \sn(x,m).
  \end{align*}
  In addition, for $m\in[0,1)$, we have the odd or even symmetry
  \begin{align*}
    \sn(-x,m) & = -\sn(x,m), \\
    \cn(-x,m) & = \cn(x,m), \\
    \dn(-x,m) & = \dn(x,m),
  \end{align*}
  and also the antiperiodicity or periodicity
  \begin{align*}
    \sn(x+2K(m),m) & = -\sn(x,m), \\
    \cn(x+2K(m),m) & = -\cn(x,m), \\
    \dn(x+2K(m),m) & = \dn(x,m).
  \end{align*}
  In particular, for $m\in[0,1)$ 
  and for any integer $\ell\in\Z$,
  \begin{align*}
    \sn\big(2\ell K(m),m \big) & = 0, \\
    \cn\big(2\ell K(m),m \big) & = (-1)^\ell, \\
    \dn\big(2\ell K(m),m \big) & = 1,
  \end{align*}
  and
  \begin{align*}
    \sn\big( (2\ell+1) K(m),m \big) & = (-1)^\ell, \\
    \cn\big( (2\ell+1) K(m),m \big) & = 0, \\
    \dn\big( (2\ell+1) K(m),m \big) & = \sqrt{1-m}.
  \end{align*}
\end{proposition}

\bibliography{elastica}

\end{document}